\documentclass[11pt]{amsart}
\usepackage{latexsym,amsfonts,amssymb}
\usepackage{amsmath}
\usepackage{amsthm}

\usepackage{psfrag}
\usepackage{epsfig}
\usepackage{pb-diagram}
\usepackage{amscd}

\usepackage{graphicx} 
\graphicspath{{./Graphics/}}
\newtheorem{lemma}{Lemma}
\newtheorem{definition}{Definition}
\newtheorem{theorem}{Theorem}
\newtheorem*{maintheorem}{Main Theorem}
\newtheorem{cor}{Corollary}
\usepackage{fullpage}
\usepackage{setspace}
\numberwithin{figure}{section}
\numberwithin{definition}{section}
\numberwithin{lemma}{section}
\numberwithin{theorem}{section}
\numberwithin{cor}{section}
\numberwithin{table}{section}
\begin{document}

\title{The Bar-Natan skein module of the solid torus and the homology of $(n,n)$ Springer varieties}
\author{Heather M. Russell}
\address{Department of Mathematics, University of Iowa, Iowa City}
\email{hrussell@math.uiowa.edu}
\date{}

\begin{abstract}
This paper establishes an isomorphism between the Bar-Natan skein module of the solid torus with a particular boundary curve system and the homology of the $(n,n)$ Springer variety. The results build on Khovanov's work with crossingless matchings and the cohomology of the $(n,n)$ Springer variety. We also give a formula for comultiplication in the Bar-Natan skein module for this specific three-manifold and boundary curve system. 
\end{abstract}
\maketitle

\section{Introduction}
In \cite{K1}, Khovanov introduces a graded ring $H^n$, and in \cite{K2}, he proves that the center of this ring is isomorphic to the cohomology of the $(n,n)$ Springer variety of complete flags in $\mathbb{C}^{2n}$ fixed by a nilpotent matrix with two Jordan blocks of size $n$. In order to establish this result, Khovanov appeals to a topological space $\widetilde{S}$ which he proves has cohomology isomorphic to the cohomology of the $(n,n)$ Springer variety.

The Bar-Natan skein module of three-manifolds arose from  Bar-NatanÕs study of the Khovanov homology of tangles and cobordisms. In our work, we consider the Bar-Natan skein module of the solid torus with boundary curve system $2n$ copies of the longitude. The main goal of this paper is to use the work done in \cite{K2} and \cite{K1} to prove that the Bar-Natan skein module in this context is isomorphic to the homology of the $(n,n)$ Springer variety. We accomplish this by working with the space $\widetilde{S}$.

Section \ref{prelim} gives introductory definitions pertaining to the Bar-Natan skein module and the theory of crossingless matchings. Section \ref{bnmod} analyzes the structure of the Bar-Natan skein module in our specific context. Section \ref{space} analyzes the structure of the space $\widetilde{S}$. Section \ref{isom} establishes the isomorphism. Section \ref{alg} discusses some additional algebraic properties of the Bar-Natan skein module.

\section{Preliminaries}\label{prelim}

\subsection{The Bar-Natan skein module}

Fix the Frobenius extension $\mathcal{A} = \mathbb{Z}[x]/(x^2)$ over $\mathbb{Z}$. A dot represents the element $x\in \mathcal{A}$, while the absence of a dot represents $1\in \mathcal{A}$. 

\begin{definition}
Let $M$ be a three-manifold, possibly with boundary. A marked surface in $M$ is a surface $S$ properly embedded in $M$ decorated with some number of dots. We regard such marked surfaces up to isotopy.  Dots can move within connected components but cannot change components.
\end{definition} 

The Bar-Natan relations on marked surfaces in a three-manifold $M$ are given below. The first two relations apply only to spheres which bound balls. The fourth relation, known as the neck-cutting relation, requires presence of a compressing disk.
\begin{align*}
 &\text{(S1)} \hspace{.5in}\raisebox{-6pt}{\includegraphics[width=.25in]{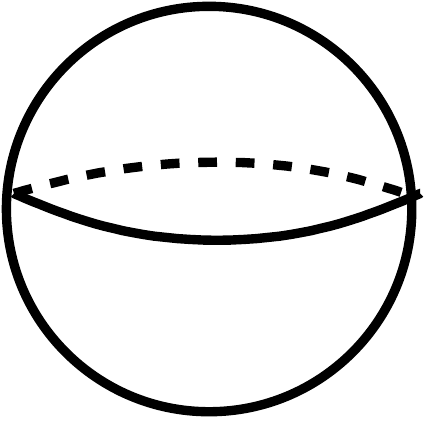}} \sqcup S  = 0 \\
 &\text{(S2)} \hspace{.5in}\raisebox{-6pt}{\includegraphics[width=.25in]{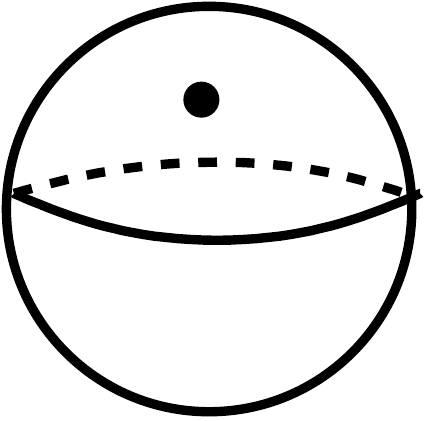}} \sqcup S - S = 0 \\
 &\text{(TD)} \hspace{.5in}\includegraphics[width=.42in]{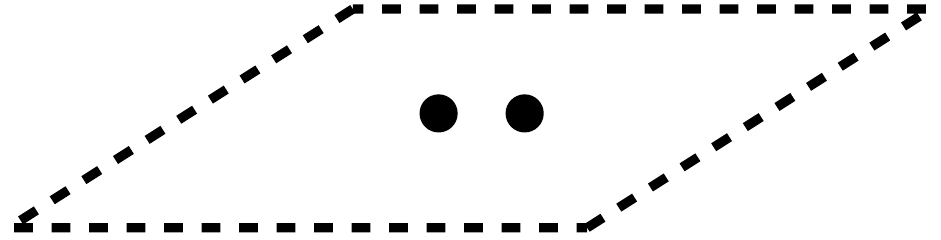} = 0 \\
 &\text{(NC)} \hspace{.5in}\raisebox{-12pt}{\includegraphics[height=.4in]{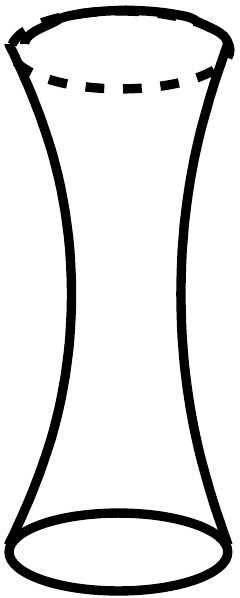}} - \raisebox{-12pt}{\includegraphics[height=.4in]{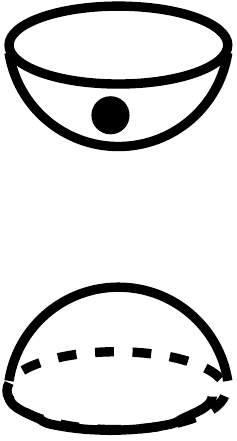}} - \raisebox{-12pt}{\includegraphics [height=.4in]{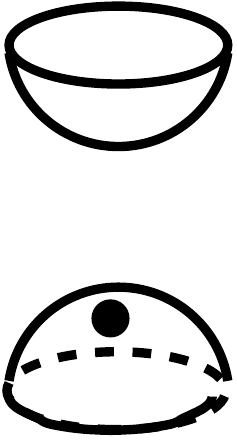}} = 0
\end{align*}
\noindent The following definition makes use of these relations.
\begin{definition}[Bar-Natan skein module]
Let $M$ be a 3-manifold. Let $c$ be a system of simple closed curves properly embedded in the boundary of $M$. Define  $\mathcal{F}(M, c)$ to be the set of isotopy classes of orientable, marked surfaces embedded in $M$ with boundary $c$. Denote the free $\mathbb{Z}$-module with basis $\mathcal{F}(M,c)$ by $\mathbb{Z} \mathcal{F}(M,c)$. Let $\mathcal{S}(M,c)$ denote the submodule of $\mathbb{Z} \mathcal{F}(M,c)$ spanned by the Bar-Natan skein relations. Finally, define $\mathcal{BN}(M,c, \mathbb{Z})$ to be the quotient of $\mathbb{Z} \mathcal{F}(M,c)$ by $\mathcal{S}(M,c)$. This is the Bar-Natan skein module of $M$ with curve system $c$.
\end{definition}

Given a marked surface $S\in \mathcal{BN}(M, c,\mathbb{Z})$, the degree of $S$, denoted $deg(S)$,  is $2d-\chi (S)$ where $d$ is the number of dots on $S$ and $\chi$ is the usual Euler characteristic. This notion of degree gives $\mathbb{Z} \mathcal{F}(M, c)$ a graded $\mathbb{Z}$-module structure.  

Since Bar-Natan relations preserve degree, this structure descends to the quotient $\mathcal{BN}(M, c, \mathbb{Z})$. In other words,  $\mathcal{BN}(M, c, \mathbb{Z})=\bigoplus _{n\in \mathbb{N} } S_n$ where each $S_n$ is generated by surfaces of degree $n$ modulo the Bar-Natan relations.
The graded rank of $\mathcal{BN}(M, c, \mathbb{Z})$ is its Poincar\'{e} polynomial $$p(M,c) = \sum_{n\in \mathbb{N} } rk(S_n)q^{n}.$$

Let $\mathcal{F}_{inc}(M, c)$ be the subset of $\mathcal{F}(M, c)$ consisting only of marked, incompressible surfaces. In \cite{AF}, it is shown that $\mathcal{BN}(M,c,\mathbb{Z})$ is generated by $\mathcal{F}_{inc}(A,c)$. In \cite{UK}, Kaiser proves that every relation between elements of $\mathcal{F}_{inc}(M,c)$ is the result of applying Bar-Natan relations to two different sequences of compressions of some surface in $\mathcal{F}(M,c)$.  

\subsection{Crossingless Matchings}
Our work and Khovanov's results from \cite{K2} and \cite{K1} rely on properties of crossingless matchings, so we recall them here. 

\begin{definition}A crossingless matching on $2n$ nodes is a collection of $n$ disjoint arcs which connect the $2n$ nodes pairwise. Let $B^n$ be the set of all crossingless matchings on $2n$ nodes. When $n$ is understood, this is shortened to $B$. 
 \end{definition}
 \begin{figure}[h]
\center{\includegraphics[height=.2in]{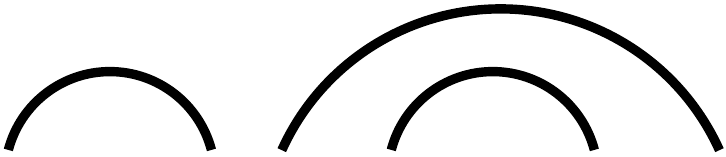}
\hspace{.3in}
\includegraphics[height=.21in]{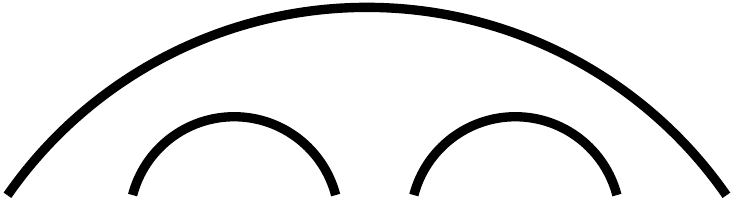}}
\caption{Two crossingless matchings on 6 nodes}
\end{figure} 
 
 Nodes are enumerated $1, \ldots , 2n$ from left to right. An arc connecting nodes $i<j$ will be denoted $(i,j)$. Given a matching $a\in B$, $(i,j)\in a$ if the arc $(i,j)$ is present in $a$. Requiring that a matching $a$ be crossingless is equivalent to there being no quadruple $i<j<k<l$ with $(i,k), (j,l) \in a$.  

\begin{definition}
An outermost arc in a matching is an arc of the form $(i, i+1)$.
\end{definition}

The following two definitions, due to Khovanov, are given in \cite{K1}.

\begin{definition}\label{order}
Define an order on $B$ as follows. 
\begin{itemize}
\item For $a,b \in B$, define the relation $a \rightarrow b$ if and only if there exist $1\leq i < j<k<l \leq n$ with $(i,j), (k,l) \in a$; $(i,l), (j,k)\in b$; and all other arcs in $a$ and $b$ identical. \\
\item Define a partial order on $B$ by $a \prec b$ if and only if there is a chain of arrows $a=a_0 \rightarrow a_1 \rightarrow \cdots \rightarrow a_{m-1} \rightarrow a_m =  b$. \\
 \item Complete this partial order to a total order, and denote it $<$. 
 \end{itemize}
 \end{definition}
 
\begin{definition}\label{dist}
Given two matchings $a, b \in B$, define the distance $d(a,b)$ between $a$ and $b$ to be the minimal length $m$ of a sequence $(a=a_0, a_1, \ldots , a_m = b)$ such that for each $i$, either $a_i \rightarrow a_{i+1}$ or $a_{i+1}\rightarrow a_i$.
\end{definition}

Given such a sequence for matchings $a$ and $b$, it is clear that $d(a_i, b) = d(a,b) - i$. Otherwise, a shorter sequence for $a$ and $b$ would exist, violating the definition of distance.  
 
\section{The solid torus with boundary curve system $c_{2n}$}\label{bnmod}

\subsection{Characterizing the generators}
Let $A$ be the standard planar annulus. The three-manifold $A\times I$ is the solid torus. Let $c_m$ be a collection of $m$ disjoint copies of the longitude of the solid torus. By isotopy, consider these curves to be embedded in $A\times \{ 0\}$, which we regard as the``bottom" of the torus. 

Recall that $\mathcal{F}_{inc}(A\times I,c_m)$, the subset of $\mathcal{F}(A\times I,c_m)$ consisting of marked, incompressible surfaces, generates the Bar-Natan skein module. A standard fact in three-manifold theory is that all closed surfaces in the solid torus are compressible and orientable, so no surface in $\mathcal{F}_{inc}(A\times I,c_m)$ has closed components. Let $S$ be a surface in $A\times I$ with boundary $c_m$ and no closed components. Then $S$ is incompressible if and only if it is the disjoint union of annuli. Since an annulus has 2 boundary components, $\mathcal{F}_{inc}(A\times I,c_m)=\emptyset$  when $m$ is odd. We refer to incompressible surfaces in $A \times I$ with boundary $c_{2n}$ as annular configurations. 

There are $C_n$ annular configurations associated to the curve system $c_{2n}$ where $C_n  = \frac{1}{n+1} \binom{2n}{n}$, the $nth$ Catalan number. According to the relation TD, each connected component of an annular configuration in $\mathcal{F}_{inc}(A\times I, c_{2n})$ can have zero or one dot. Considering all possibilities, $\mathcal{F}_{inc}(A\times I, c_{2n})$ has $$C_n + \binom{n}{2} C_n + \binom {n}{3} C_n + \cdots + \binom{n}{n-1}C_n + \binom{n}{n}C_n = C_n\sum_{i=0}^{n} \binom{n}{i}= 2^nC_n$$ elements. 

Another common characterization of $C_n$ is the number of crossingless matchings on $2n$ nodes. There is a one-to-one correspondence between annular configurations with boundary $c_{2n}$ and crossingless matchings on $2n$ nodes. Given a matching, rotating around a disjoint, vertical axis will generate the associated annular configuration. By allowing arcs in a matching to carry dots, this correspondence extends to all marked annular configurations.

 \subsection{Characterizing the relations}
 
The goal of this section is to present a simple generating set for all relations between marked, incompressible surfaces as elements of the Bar-Natan skein module. 

\begin{definition}
Two annuli in a configuration are adjacent if there exists an arc in their complement connecting them and intersecting no other annuli in the configuration. 
\end{definition}

After tubing two adjacent annuli, the core of that tube gives a trivial compressing disk for the resulting surface; this is trivial in the sense that compression along that disk recovers the initial surface.  Relations between marked, incompressible surfaces arise in the presence of two or more non-isotopic compressing disks.  
 
The simplest case where relations occur is $n=2$. The curve system $c_{4}$ bounds two different incompressible surfaces in $A\times I$: the unnested and nested configurations. 
 
 \begin{figure}[h]
 \center{\includegraphics[width=2.1in]{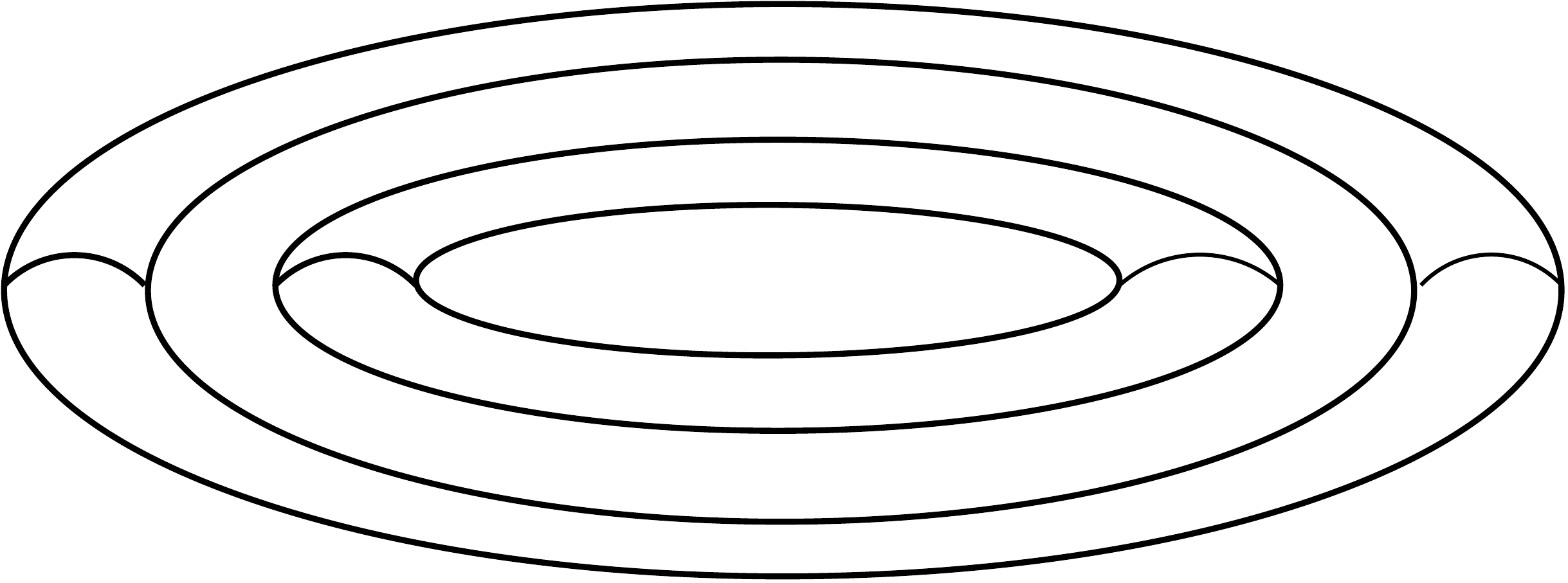}
 \hspace{.5in}
 \includegraphics[width=1.95in]{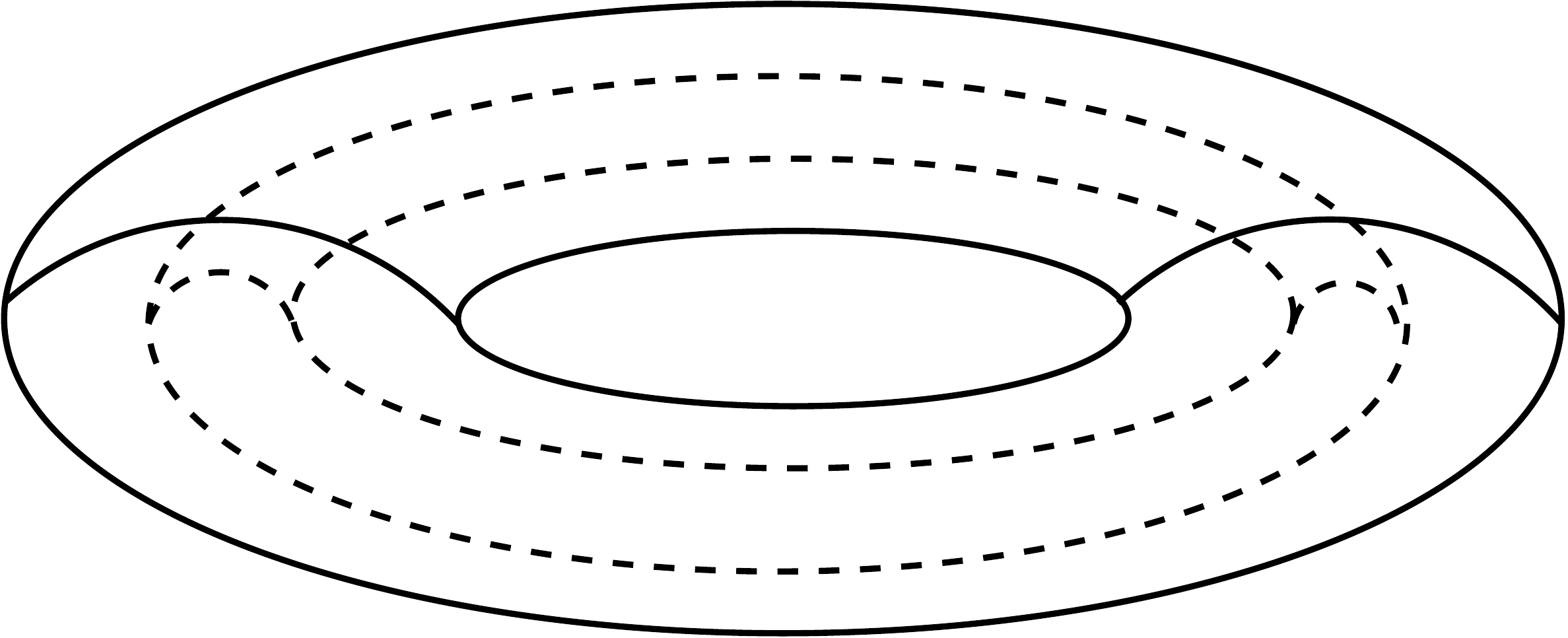}}
 \caption{The unnested and nested configurations}
 \end{figure}
 
The surface obtained by tubing the annuli in the unnested configuration using a standard, unknotted tube, has two non-isotopic classes of compressing disks (see Figure \ref{disks}).
 
 \begin{figure}[here]
 \center{\includegraphics[width=2in]{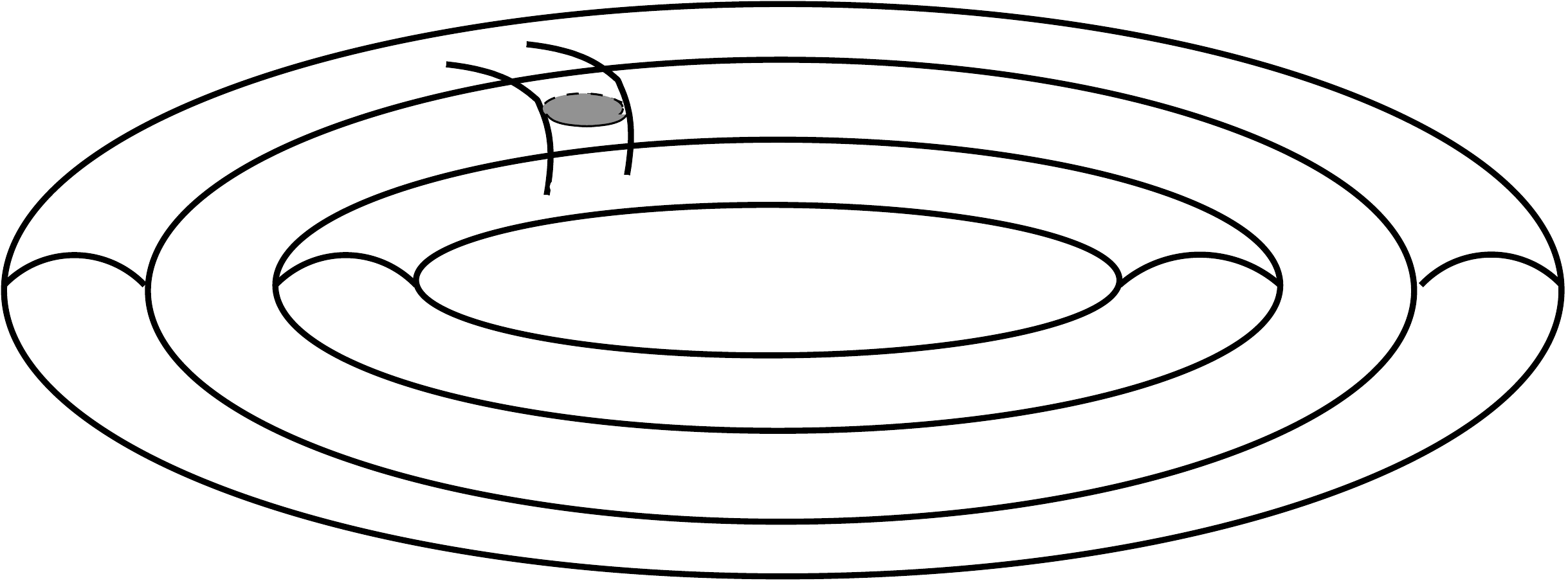}
 \hspace{.5in}
\raisebox{-3pt}{ \includegraphics[width=2.16in]{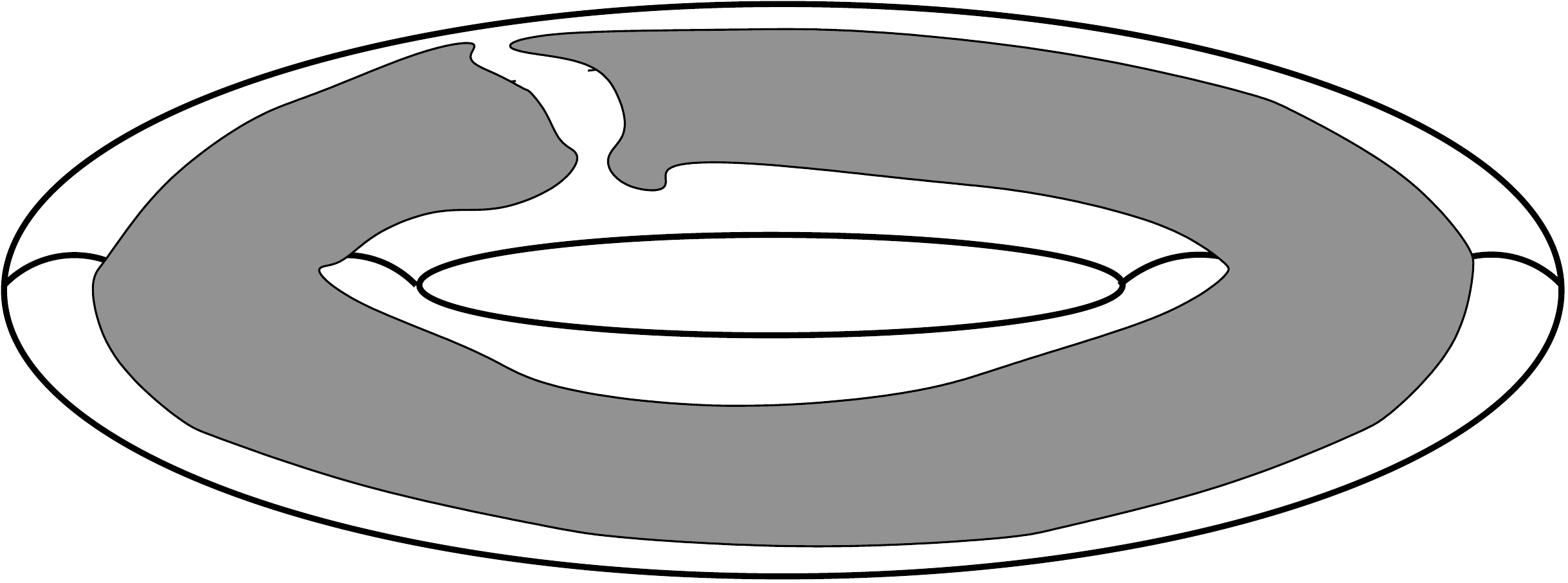}}}
 \caption{Trivial and nontrivial compressing disks.}
 \label{disks}
 \end{figure}
 
As previously stated, compression along the trivial disk gives the unnested configuration. Compression along the nontrivial disk, however, gives the nested configuration.  Applying NC to both compressing disks in the tubed, unnested configuration with no dots results in the following relation, called a Type I relation.
\begin{equation}\label{type1}
 \raisebox{-5pt}{\includegraphics[width=1in]{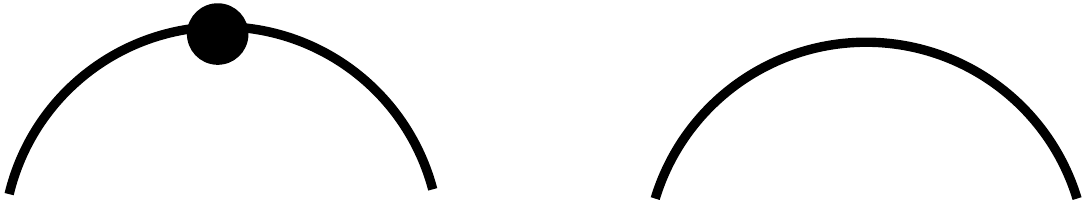}} + \raisebox{-5pt}{\includegraphics[width=1in]{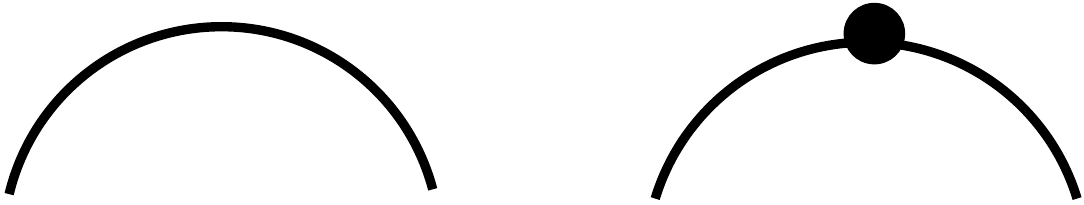}} \hspace{.1in} = \hspace{.1in} \raisebox{-10pt}{ \includegraphics[width=1in]{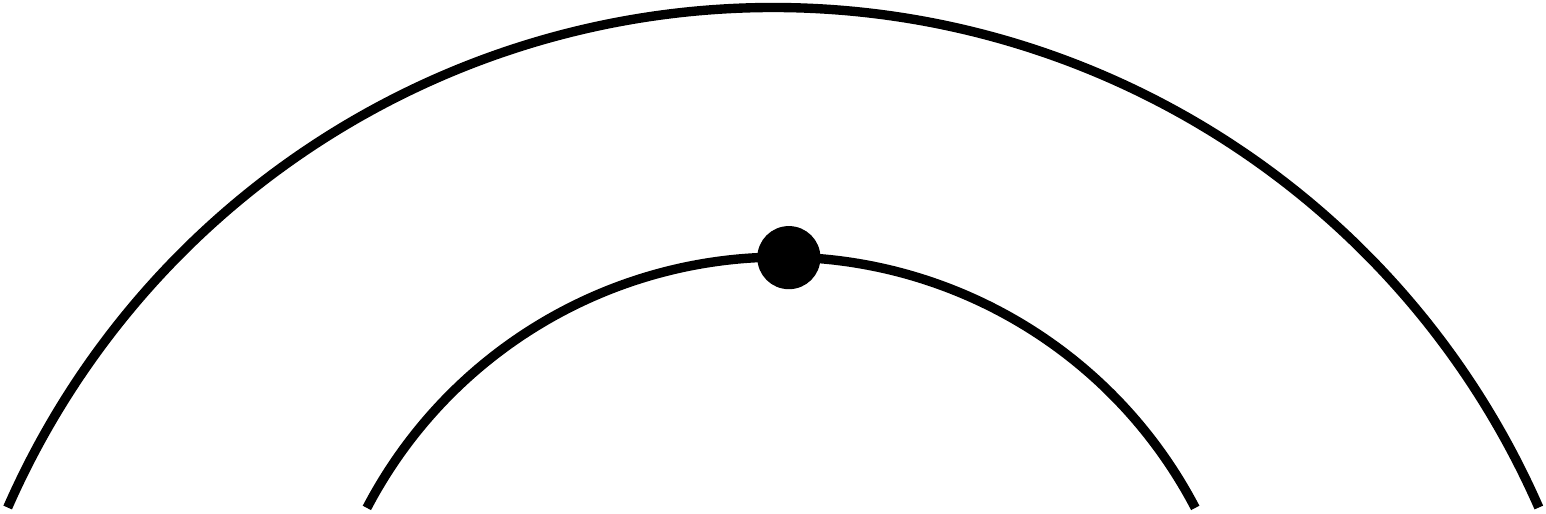}} + \raisebox{-10pt}{\includegraphics[width=1in]{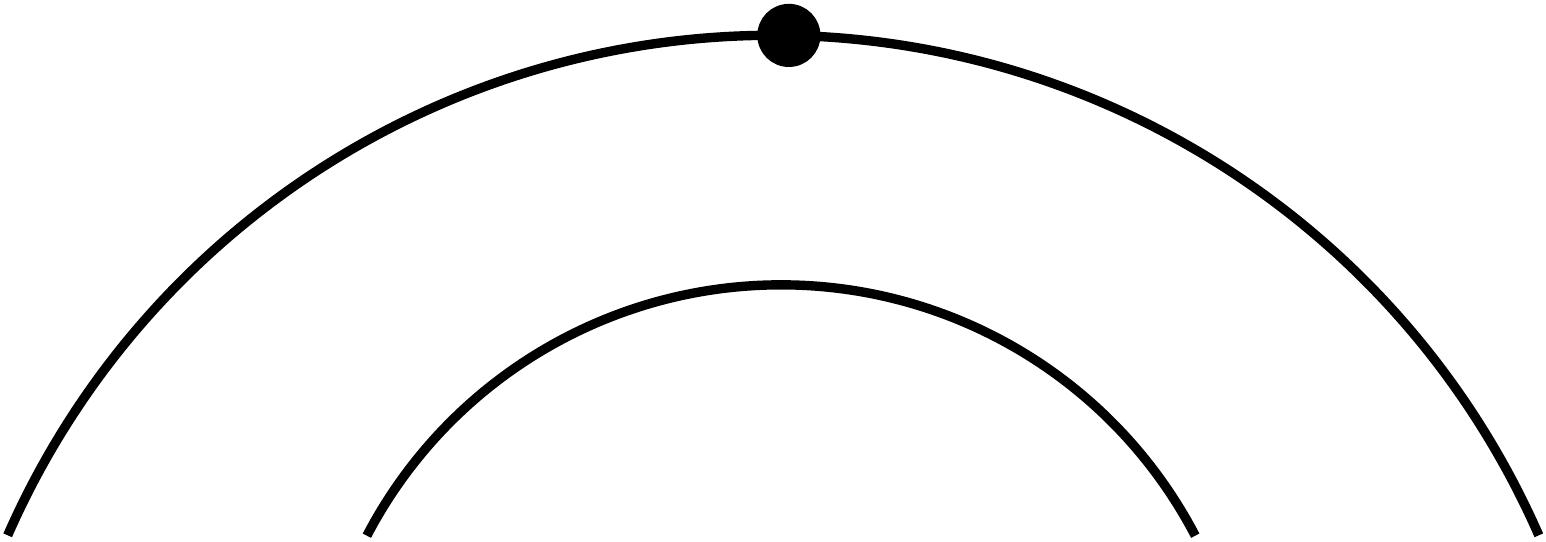}}
\end{equation}
  
Performing the same compressions in the unnested, tubed configuration with one dot results in the following relation, called a Type II relation.
\begin{equation*}
\raisebox{-5pt}{\includegraphics[width=1in]{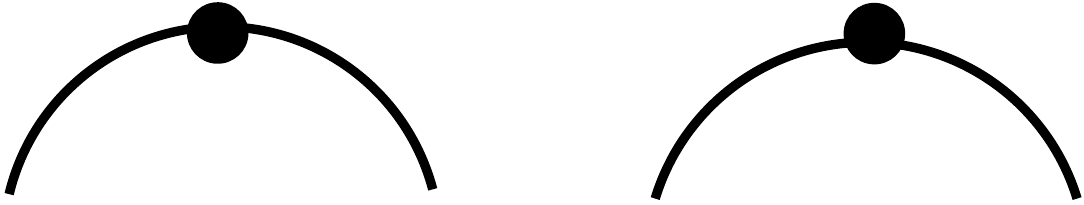}} + \raisebox{-5pt}{\includegraphics[width=1in]{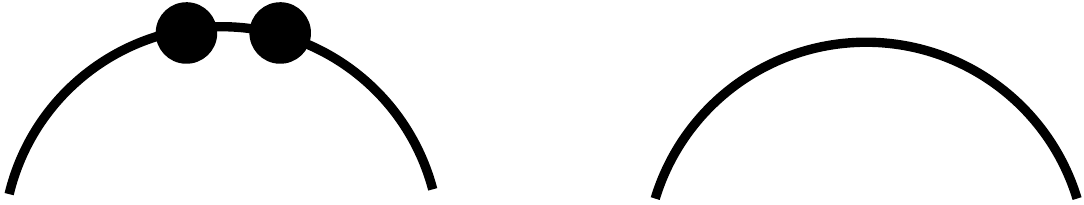}} \hspace{.1in} = \hspace{.1in} \raisebox{-10pt}{ \includegraphics[width=1in]{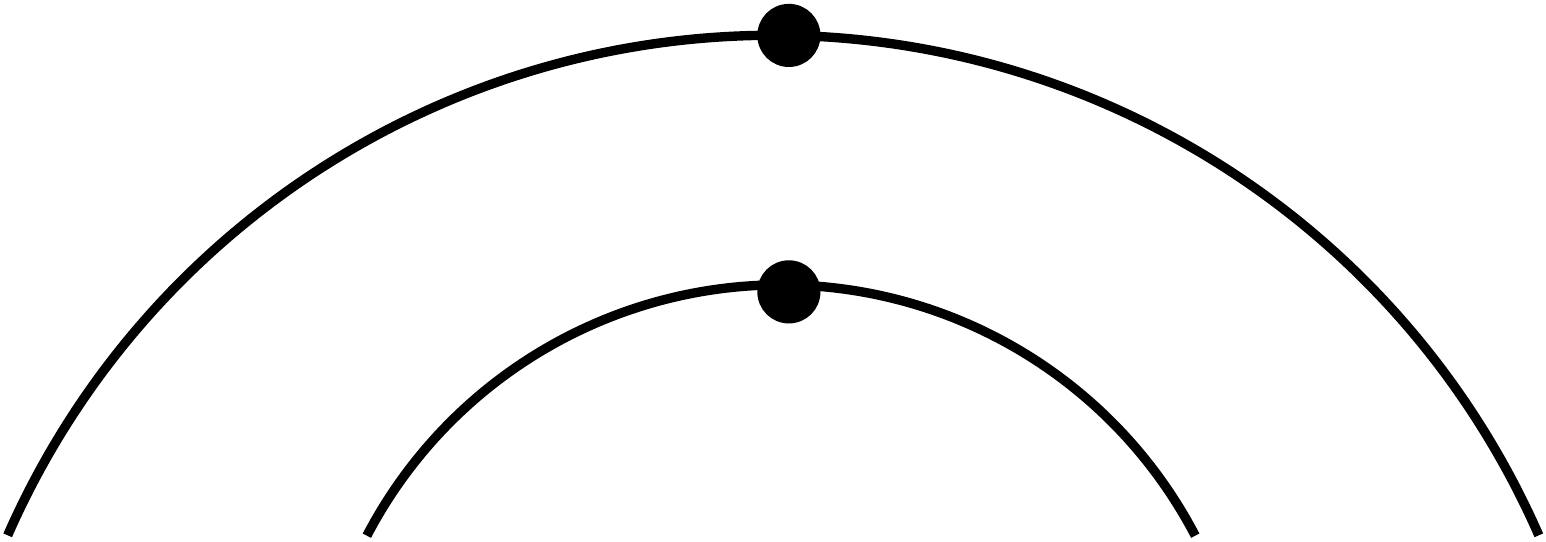}} + \raisebox{-10pt}{\includegraphics[width=1in]{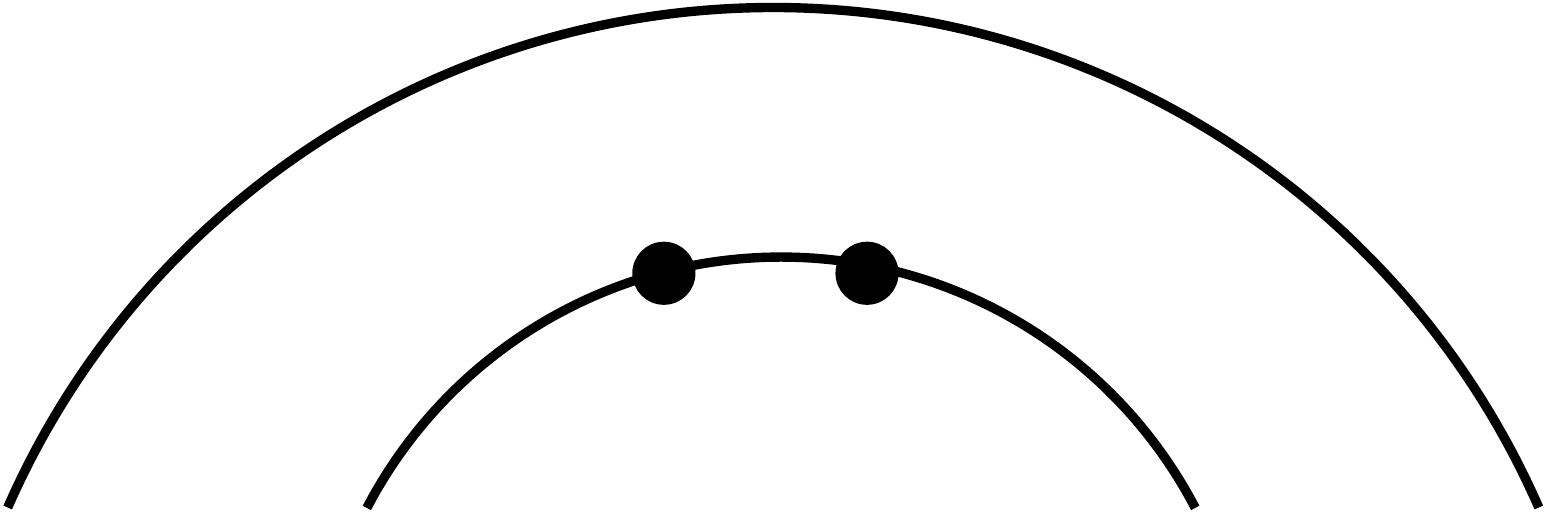}}
\end{equation*}

Applying TD, this Type II relation can be simplified to
\begin{equation}\label{type2}
\raisebox{-5pt}{\includegraphics[width=.9in]{unnest3.pdf}} \hspace{.1in} = \raisebox{-10pt}{ \includegraphics[width=.9in]{nest3.pdf}}
\end{equation}

For each pair of adjacent annuli in an annular configuration, there are Type I and Type II relations obtained by generalizing Equations (\ref{type1}) and (\ref{type2}) above. 

The remainder of this section is dedicated to proving the following result

\begin{theorem}\label{reltypes}
Any relation on incompressible, marked surfaces can be written as a linear combination of Type I and Type II relations.
\end{theorem}
Let $S$ be a marked, connected surface in $\mathcal{F}(A\times I, c_{2n})$.  Let $D=(D_1, D_2, \ldots, D_k)$ and $D'=(D_1', \ldots, D_l')$ be two families of compressing disks for $S$ that  produce a relation between elements of $\mathcal{F}_{inc}(A\times I, c_{2n})$.  In other words, both sequences of compressions $D$ and $D'$ terminate at incompressible surfaces with boundary $c_{2n}$ and thus are configurations of $n$ annuli. Let $S=S_0$; let $S_1$ be the surface obtained by compressing $S_0$ along $D_1$.  In general, we say $D_i$ is a compressing disk for the intermediate surface $S_{i-1}$. Let the intermediate surfaces for the other family of compressions be denoted $S_i'$. 

By sequentially applying NC to each disk in $D$ and reducing by the other Bar-Natan relations, a linear combination of elements in $\mathcal{F}_{inc}(A\times I, c_{2n})$ is produced. In particular, this linear combination will consist of different markings of the annular configuration $S_k$. Performing the same process for the family $D'$ produces a second linear combination of markings of the annular configuration $S_l'$. Because $D$ and $D'$ are families for the same initial surface, the two expressions are equal. By the result from \cite{UK} stated earlier, every relation among marked, incompressible surfaces can be obtained from two such families of compressions.

\begin{lemma}
All relations can be generated by families $D$ and $D'$ whose prescribed compressions never create closed components.
\end{lemma}

\begin{proof}
WLOG, assume that  compression of $S_i$ along $D_{i+1}$  yields a closed component. Write $S_{i+1} = S_{i+1}^1 \sqcup S_{i+1}^2$ where $S_{i+1}^2$ is the closed component. 

We argue that when $S_{i+1}^2$ is not a torus, this compression was superfluous. If $S_{i+1}^2$ is a torus, it  can be replaced with a different, non-separating compression.

Via NC, if $S_i$ carries no dots initially
\begin{equation}\label{closedcomp}
S_{i} = \dot{S}_{i+1}^1 \sqcup S_{i+1}^2 + S_{i+1}^1 \sqcup \dot{S}_{i+1}^2
\end{equation}

\noindent where the dots indicate which components carry dots. If $S_i$ has one dot initially, NC yields

\begin{equation}\label{closedcompdot}
\dot{S_i} = \dot{S}_{i+1}^1 \sqcup \dot{S}_{i+1}^2
\end{equation}

If $S_{i+1}^2$ is a sphere, Eq. \ref{closedcomp} reduces to  $S_{i} =  S_{i+1}^1$ via  $S1$ and $S2$. Eq. \ref{closedcompdot} reduces to $\dot{S_i} = \dot{S}_{i+1}^1$.  In either case, compression along $D_{i+1}$ contributes nothing and was not necessary. 

If $S_{i+1}^2$ is a torus, Eq. \ref{closedcomp} reduces to $S_{i} = \dot{S_{i+1}^1} \sqcup S_{i+1}^2$, via NC and TD. Using NC once again,  the torus evaluates to 2, so Eq. \ref{closedcomp} further reduces to $S_{i} = 2\dot{S_{i+1}^1}$. Replacing $D_{i+1}\in D$ with the disk shown in Figure \ref{tor} will  produce the same relation without creation of closed components. Since a dotted torus evaluates to 0 by applying NC, Eq. \ref{closedcompdot} reduces to $\dot{S_i} = 0$ .

Finally, if $S_{i+1}^2$ has higher genus, both $S_{i+1}^2$ and $\dot{S_{i+1}^2}$ evaluate to 0 by NC and TD. This means $S_i = 0$, and $D$ and $D'$ do not yield a relation. 
\end{proof}
\begin{figure} [h]
\center{\includegraphics[height=.75in]{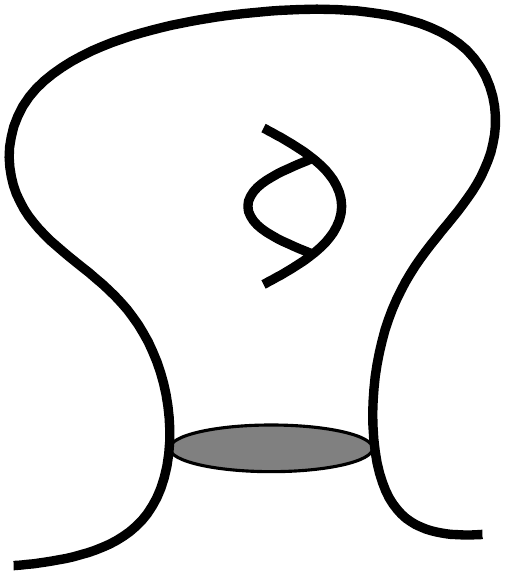} \hspace{.75in} \includegraphics[height=.75in]{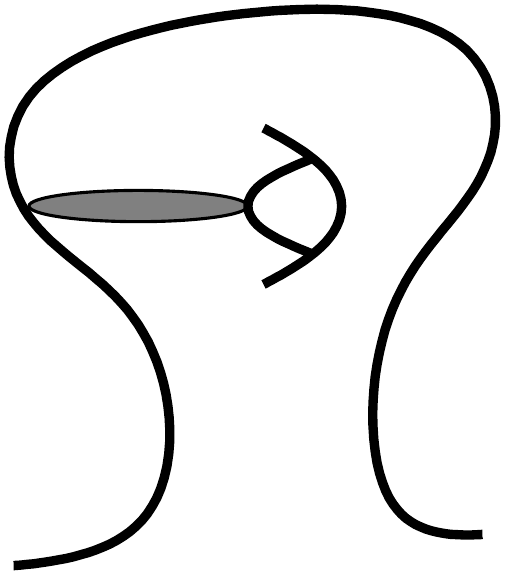}}
\caption{An alternative compressing disk}
\label{tor}
\end{figure}
Since compressions involving closed components have been shown to be unnecessary, assume $D$ and $D'$ contain no such compressing disks.

\begin{lemma}\label{numcomp}
Given the two families of compressing disks, $D = (D_1, D_2, \ldots, D_k)$ and $D'= (D_1', \ldots, D_l')$, $k=l$. In other words, the two families have the same number of disks.
\end{lemma}

\begin{proof}
The terminal surfaces $S_k$ and $S_l'$ are annular configurations, so $\chi (S_k) = \chi (S_l') = 0$. Since each compression increases Euler characteristic by 2, $\chi (S) + 2k = \chi (S_k)$, and $\chi(S) + 2l = \chi (S_l')$. Therefore, $\chi (S) + 2k =\chi(S) + 2l$, and $k=l$. 
\end{proof}

By definition, non-separating compressions do not change the number of connected components of a surface. Separating compressions increase the number of connected components by one. Since $S_k$ and $S_l'$ are comprised of $n$ annuli and the initial surface $S$ is connected, the families $D$ and $D'$ must consist of at least $n-1$ separating compressions. Thus, $k\geq n-1$. 

Every time a compression is performed and NC is applied, one additional dot is added to some component of each surface in the sum being generated. Application of NC to all disks in the family $D$ produces a sum of copies of the annular configuration $S_k$ where each term is marked with the same number of dots distributed across the components of $S_k$.  If the initial surface $S$ is dotted, each term in the final relation will have $k+1$ dots. If $S$ is undotted, each term will have $k$ dots. 

Since $k \geq n-1$, a dotted initial surface $S$ results in a final linear combination where each term will have at least $n$ dots . If $k>n-1$, there will be more than $n$ dots on each surface in the relation, so each term in the relation will have some component with two or more dots. By TD, such a  relation has no nonzero terms. If $k=n-1$, the resulting relation is, on each side, the sum of all possible ways to arrange $n-1$ dots on the $n$ annuli of some configuration, one of which already has a dot. There is only one non-zero way to accomplish this, so the resulting relation has one non-zero term on each side, the configuration with exactly one dot on each component.

\begin{lemma} \label{alldot}
All configurations of $n$ annuli with each annulus dotted are equivalent via Type II relations.
\end{lemma}

\begin{proof} 
Given $n\in \mathbb{N}$, we call the configuration with associated matching $(1,2), (3,4), \ldots , (2n-1, 2n)$ the outermost configuration in $B^n$. We will show that every all-dotted annular configuration of $n$ annuli is equivalent to the all-dotted, outermost configuration via Type II relations.

\textbf{Base Case}: Let  $n=2$. The all-dotted, nested configuration in $B^2$ is equivalent to the all-dotted outermost configuration via the basic Type II relation.

\textbf{Inductive Argument}: Assume that for $m<n$, all configurations of $m$ annuli with each annulus dotted are equivalent to the all-dotted, outermost configuration via Type II relations.  Let $a\in B^n$, and consider the all-dotted configuration corresponding to $a$.

If $a$ does not include the arc $(1,2n)$, then it has some collection of arcs: $(1, i_1), (i_1+1, i_2) , \ldots, (i_{t}+1, 2n)$ with $i_1\neq 2n$.

\begin{figure}[htp]
\begin{picture}(450,10)(-150,0)
\includegraphics[width=2in]{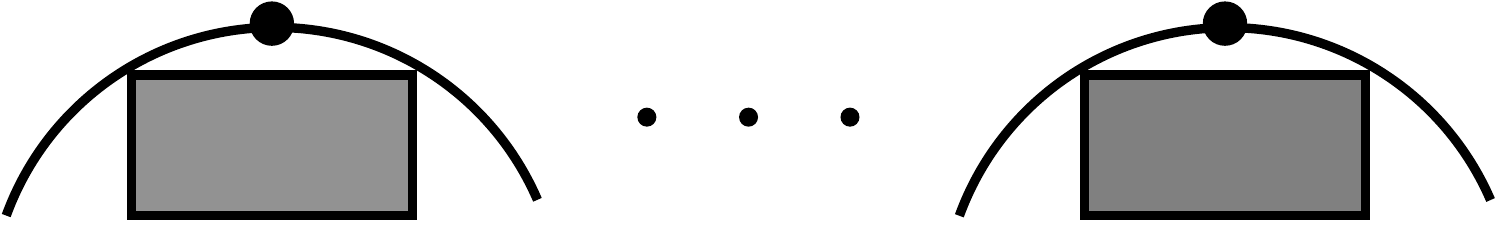}
\put(-145,-10){1}
\put(-95, -10){$i_1$}
\put(-62,-10){$i_{t}+1$}
\put(-3,-10){$2n$}
\end{picture}
\end{figure}

By the inductive argument, the all-dotted configuration associated to $a$ is equivalent to the all-dotted configuration associated to the matching with $(1,2), (3,4), \ldots (i_{1}-1, i_1)$ and all other arcs identical to $a$. We can repeat this process for each arc $(i_{s}+1, i_{s+1})$ and the arcs below it leaving everything else fixed. Repeating this process $t+1$ times, we arrive at the all-dotted, outermost configuration.

\begin{figure}[htp]
\begin{picture}(430,10)(-140,0)
\includegraphics[width=2in]{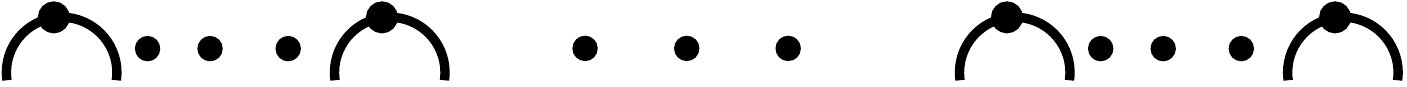}
\put(-147,-10){1}
\put(-100,-10){$i_1$}
\put(-58,-10){$i_{t}+1$}
\put(-3,-10){$2n$}
\end{picture}
\end{figure}

If $a$ includes the arc $(1,2n)$, it also includes an arc $(2, t)$ where $t< 2n$. 

\begin{figure}[htp]
\begin{picture}(450,25)(-150,0)
\includegraphics[width=2in]{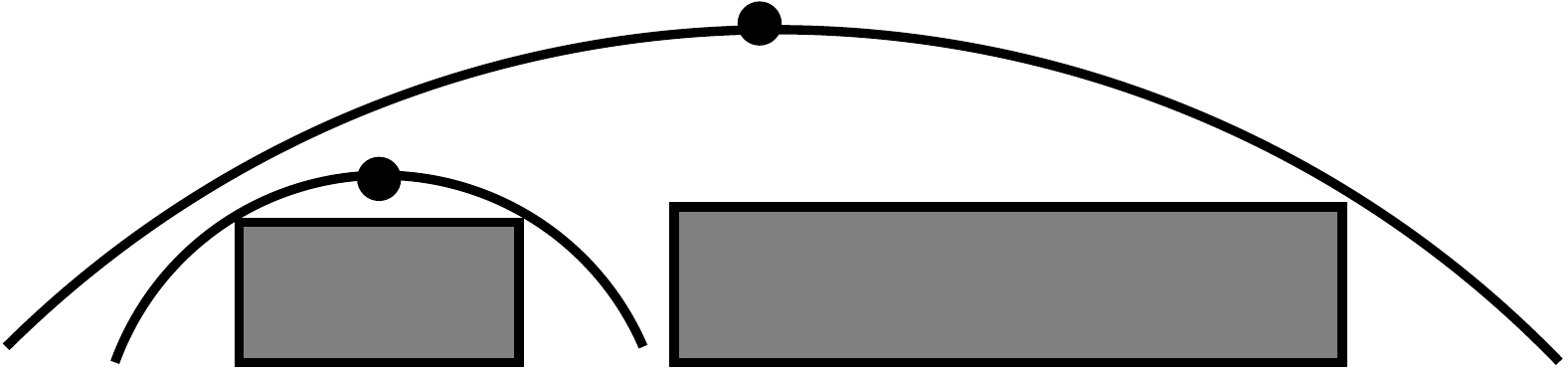}
\put(-150,-10){1}
\put(-137, -10){2}
\put(-86,-10){$t$}
\put(-5,-10){$2n$}
\end{picture}
\end{figure}

By a single Type II relation, the all-dotted configuration associated to $a$ is equivalent to the all-dotted configuration associated to the matching with the arcs $(1,2)$, $(t, 2n)$, and all other arcs identical to $a$. 

\begin{figure}[htp]
\begin{picture}(450,30)(-150,0)
\includegraphics[width=2in]{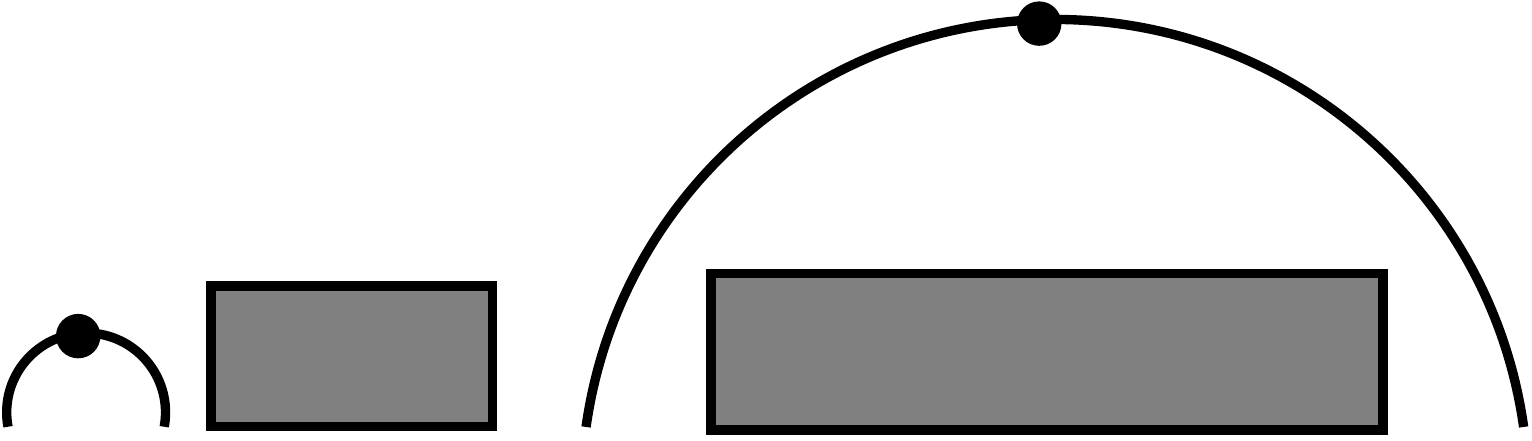}
\put(-147,-10){1}
\put(-132, -10){2}
\put(-92,-10){$t$}
\put(-4,-10){$2n$}
\end{picture}
\end{figure}

Appealing to the previous case, we conclude that this all-dotted configuration is equivalent to the all-dotted, outermost configuration.
\end{proof}
\begin{cor}\label{cor}
Any two configurations of $n$ annuli with the same single undotted annulus are equivalent via Type I and Type II relations.
\end{cor}
By Lemma \ref{alldot},  every relation that begins with a dotted surface can be written in terms of Type I and Type II relations. 

If the initial surface was undotted and $k>n$, each term in the resulting relation has more than $n$ dots distributed across $n$ components. This forces one or more components to have two dots, which means this relation is comprised of all zero terms. 

The only two remaining cases are when $S$ is undotted, and $k=n$ or $n-1$.

If $k=n$, the families $D$ and $D'$ consist of $n-1$ separating disks and 1 non-separating disk. Applying NC to the non-separating disk produces a factor of 2 in the terms that had no dots on the component where compression occured. It will force all other terms to be zero. Thus, when $k=n$, any system of $n$ compressions leads to twice some annular configuration with all components dotted.

All such configurations are equivalent by Lemma \ref{alldot}, so the case when $k=n$ can be expressed as a linear combination of Type II relations.

If $k=n-1$, then every disk is separating. Each time a compression occurs, one new connected component and one additional dot are produced. Application of NC to all $k$ compressions results in a linear combination of all possible ways to arrange $n-1$ dots on $n$ annuli. The resulting relation will therefore have $n = \binom{n}{n-1}$ terms on each side each appearing with a coefficient of 1.  

Appealing to Corollary \ref{cor}, we can group equivalent configurations with one undotted annulus into classes according to the one undotted arc $(i,j)$ in their associated matchings. An element of this class will be denoted by $\overline{(i,j)}$.

\begin{lemma} \label{altsum}
Given any $n\in \mathbb{N}$, a configuration $\overline{(i,j)}$ can be written, using Type I and Type II relations, as the alternating sum $\overline{(i,j)} = \overline{(i,i+1)} - \overline{(i+1, i+2)} + \cdots - \overline{(j-2, j-1)}+ \overline{(j-1, j)}.$ 
\end{lemma}
\begin{proof}
\underline{Base Case}: For $n=2$, this is just the basic Type I relation.

\underline{Inductive Argument}: Assume for $m<n$, every configuration $\overline{(i,j)}$ of $m$ annuli can be written $\overline{(i,j)} = \overline{(i,i+1)} - \overline{(i+1, i+2)} + \cdots + \overline{(j-1, j)}$. Let $\overline{(i,j)}$ be a configuration of $n$ annuli. 

If $(i,j) \neq (1,2n)$, then $j-i=2m-1$ where $m<n$. 
\begin{figure}[htp]
\begin{picture}(430,35)(-140,4)
\includegraphics[width=2in]{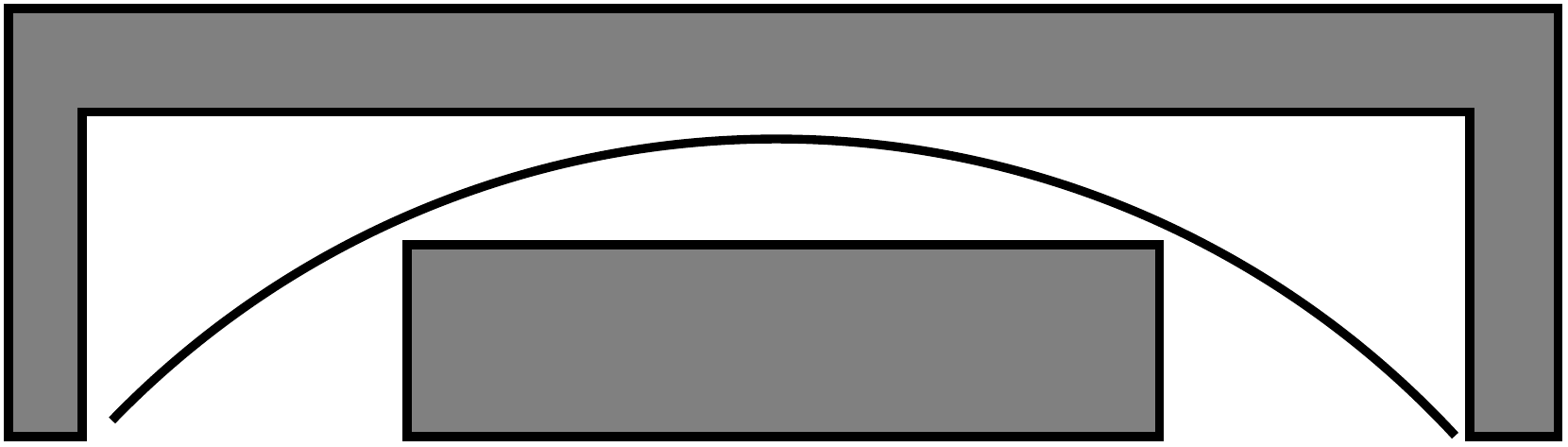}
\put(-135,-10){$i$}
\put(-14,-10){$j$}
\end{picture}
\end{figure}
Since the arc $(i,j)$ and those below it constitute a matching on $2m$ nodes, we may apply the inductive assumption leaving all arcs outside of $(i,j)$ fixed to get the desired equation.

If $(i,j) = (1,2n)$, then we may assume $\overline{(i,j)}=\overline{(1,2n)}$ is the configuration associated to the matching with undotted arc $(1,2n)$ and dotted arcs $(2,3), \ldots, (2n-2, 2n-1)$.

\begin{figure}[htp]
\begin{picture}(430,30)(-140,-1)
\includegraphics[width=2in]{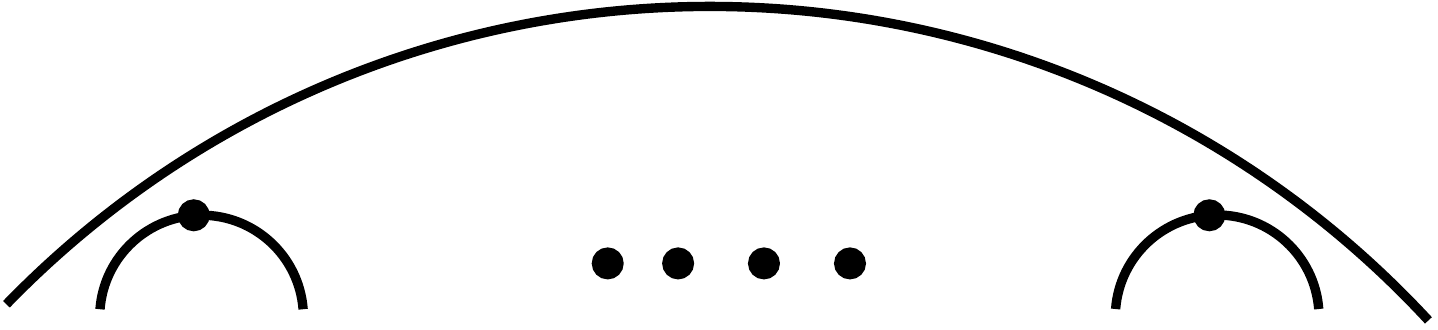}
\put(-145,-10){1}
\put(-8,-10){$2n$}
\end{picture}
\end{figure}

Via a Type I relation, we have $\overline{(1,2n)} = \overline{(1,2)} - \overline{(2,3)} + \overline{(3,2n)} .$ Applying the inductive assumption, we get $$\overline{(1,2n)} = \overline{(1,2)} - \overline{(2,3)} + \Big( \overline{(3,4)} - \overline{(5,6)} + \cdots +\overline{(2n-1, 2n)}\Big) .$$

So, given any $n\in \mathbb{N}$ and any configuration $\overline{(i,j)}$ of $n$ annuli, it can be written as desired using Type I and Type II relations.
\end{proof}

\begin{lemma}\label{sum}
Given any configuration of $n$ annuli, the sum of all possible ways to place $n-1$ dots on that configuration reduces, via Type I and Type II relations, to $\overline{(1,2)} + \overline{(3,4)} + \cdots + \overline{(2n-1, 2n)}$.
\end{lemma}
\begin{proof}
Let $a\in B^n$ be the matching associated to some annular configuration. The only nonzero way to place $n-1$ dots on $n$ annuli is to place 1 dot on all but one annulus. The sum of all possible ways to do this can be written $$\sum_{(e_1, e_2)\in a} \overline{(e_1, e_2)}$$ where the sum is over all arcs $(e_1, e_2)$ in the matching $a$.

\underline{Base Case}: If n=2, there are two configurations, and each configuration has two markings with one dot. The unnested configuration has $\overline{(1,2)}$ and $\overline{(3,4)}$, while the nested configuration has $\overline{(1,4)}$ and $\overline{(2,3)}$. For the unnested configuration, the sum $\overline{(1,2)} + \overline{(3,4)}$ is already in the appropriate form.

 Applying Lemma \ref{altsum},  $\overline{(1,4)} = \overline{(1,2)} - \overline{(2,3)} + \overline{(3,4)}$. Thus, $$\overline{(1,4)}+\overline{(2,3)} = (\overline{(1,2)}-\overline{(2,3)}+\overline{(3,4)}) + \overline{(2,3)} = \overline{(1,2)}+ \overline{(3,4)},$$ and the base case is proven.
 
 \underline{Inductive Argument}: Assume for $m<n$ and any configuration of $m$ annuli, this lemma is true. Take some configuration of $n$ annuli and its associated matching $a$. If $(1, 2n)\notin a$, then $a$ has some collection of arcs: $(1, i_1), (i_1+1, i_2) , \ldots, (i_t+1, 2n)$. We have the following,
 \begin{equation}\label{fatty}
 \sum_{(e_1,e_2)\in a} \overline{(e_1,e_2)} = \sum_{\begin{subarray}{c}
 (e_1,e_2)\in a \\
 1\leq e_1 < e_2\leq i_1
 \end{subarray}} \overline{(e_1,e_2)} + \cdots + \sum_{\begin{subarray}{c}
 (e_1,e_2)\in a \\
 i_t+1\leq e_1 < e_2\leq 2n
 \end{subarray}}\overline{(e_1,e_2)}.
 \end{equation}
 
 Applying the inductive assumption, Eq. \ref{fatty} becomes
  \begin{equation}
    \sum_{(e_1,e_2)\in a} \overline{(e_1,e_2)} = \Big( \overline{(1, 2)} + \cdots + \overline{(i_1-1, i_1)}\Big)+ \cdots + \Big( \overline{(i_{t}+1, i_{t+1})} + \cdots + \overline{(2m-1, 2m)}\Big) 
    \end{equation}
 
 If the arc $(1,2n)\in a$, then, applying the inductive assumption to the arcs below $(1, 2n)$, we have 
 $$\sum_{(e_1, e_2)\in a} \overline{(e_1,e_2)}  = \overline{(1,2n)} + \Big( \overline{(2,3)} + \overline{(4,5)} + \cdots + \overline{(2n-2, 2n-1)}\big).$$ By Lemma \ref{altsum}, this can be rewritten as
 \begin{eqnarray*} 
\sum_{(e_1,e_2)\in a} \overline{(e_1,e_2)} & = & \big( \overline{(1,2)} - \overline{(2,3)} + \cdots + \overline{(2n-1, 2n)}\big) +  \big( \overline{(2,3)} + \overline{(4,5)} + \cdots + \overline{(2n-2, 2n-1)}\big)\\
& =& \overline{(1,2)}+\overline{(3,4)}+\cdots + \overline{(2n-1, 2n)}.
\end{eqnarray*}
So, the lemma is true for all configurations of any number of annuli.
\end{proof}

By the Lemma \ref{sum}, any relation coming from an initially undotted surface where the number of compressions is $k=n-1$ are expressible in terms of Type I and Type II relations. Having considered all cases, we conclude that all relations on connected surfaces coming from families $D$ and $D'$ can be written in terms of Type I and Type II relations. 

Assume that any relation coming from a surface with less than $j$ components can be written as a linear combination of Type I and Type II relations. Consider a surface $S\in \mathcal{F}(A\times I, c_{2n})$ with $j$ connected components. Write $n=\sum_{i=1}^j n_i$ where $c_{2n_i}$ is the boundary of the $i^{th}$ component of $S$. Let $D=(D_1. \ldots, D_k)$ and $D'=(D_1', \ldots, D_l')$ be two families of compressions for $S$ that give a relation on incompressible, marked surfaces. 

After performing all compressions prescribed by $D$ and $D'$, apply the connected case to rewrite the contribution of the $j^{th}$ connected component in terms of Type I and Type II relations. Then, apply the inductive assumption to rewrite the contributions from the remaining $j-1$ components in terms of Type I and Type II relations. The result is an expression written entirely in terms of Type I and Type II relations. Thus, Theorem \ref{reltypes} is proven.
    
\section{The space $\widetilde{S}$}\label{space}

\subsection{Construction}\label{const}
Let $a\in B^n$, and let $S^2$ denote the standard two-sphere.  Let $(S^2)^{2n}$ be the cartesian product of $2n$ copies of $S^2$. Define $S_a$ as follows, $$S_a=\{ (x_1, \ldots, x_{2n})\in (S^2)^{2n} : x_i=x_j \text{ if } (i,j)\in a\}.$$ Notice $S_a$ is diffeomorphic to $(S^2)^{n}$.  

 Let $S_{<a} = \bigcup_{b<a} S_b$ and $S_{\leq a} = \bigcup_{b\leq a} S_b$. Finally, let $\widetilde{S} = \cup_{a\in B} S_a$.  

The standard cell decomposition of $S_a$, which we will discuss further in the next section, has only even dimensional cells, so all boundary maps in the chain complex for homology of $S_a$ are zero. Thus, $H_*(S_a)=C_*(S_a)$ and $H^*(S_a) = Hom(C_*(S_a), \mathbb{Z}) \cong C_*(S_a)$. Furthermore, the homology groups are free abelian with bases given by the cells. In \cite{K1}, another even-dimensional cell-decomposition for $S_a$ is given. This decomposition has the property that both $S_a\cap S_b$ for $b<a$ and $S_{<a}\cap S_a$ are subcomplexes.  The homologies and cohomologies of these spaces are, therefore, also finitely generated free abelian. 

The cohomology of $\widetilde{S}$ can be calculated by recursively applying the Mayer-Vietoris sequence. At each iteration, let $X=S_{\leq a}$, $A= S_a$, $B=S_{<a}$, and $Y=S_{<a}\cap S_a$. So, $X=A\cup B$, and $Y=A\cap B$. Since $H^{2k+1}$ is zero for $A, B, Y$, and $X$ for all $a\in B$ for all $k\in \mathbb{Z}$, the Mayer-Vietoris sequence reduces to  a collection of short exact sequences of the form: 
$$0\rightarrow H^{2k}(S_{\leq a}) \rightarrow H^{2k}(S_{a}) \oplus H^{2k}(S_{<a}) \rightarrow H^{2k}(S_{<a}\cap S_a) \rightarrow 0.$$
These sequences are split short exact since $H^{2k}(S_{<a}\cap S_a)$ is finitely generated free abelian and hence a projective $\mathbb{Z}$-module. Because  $H^{2k}(S_{\leq a})$ is isomorphic to a submodule of $H^{2k}(S_{a}) \oplus H^{2k}(S_{<a})$, it is also free abelian. Furthermore, $$rk(H^{2k}(S\leq a)) + rk(H^{2k}(S_{<a}\cap S_a)) = rk(H^{2k}(S_a)) + rk(H^{2k}(S_{<a})).$$ At the final application of Mayer-Vietoris, $S_{\leq a} = \tilde{S}$, so $H^*(\tilde{S})$ is finitely generated free-abelian.

\subsection{Understanding cell decomposition and inclusion}\label{inc}

The standard CW-decomposition of the two sphere is given by some point $\{p \} \in S^2$ and the two-cell $c = S^2 - \{ p\}$. Therefore, the total homology of $S^2$ is $H_*(S^2) = H_0(S^2) \oplus H_2(S^2)$ where $H_0(S^2) = \mathbb{Z}\{ p\}$ and $H_2(S^2) = \mathbb{Z}c$.  This standard structure can be used to endow $S_1^2\times \cdots \times S_n^2$ with the cartesian product cell decomposition and basis for homology. 

Since there is no torsion in the homology of $S^2$, the K\"{u}nneth Formula gives us $$H_n(S_1^2 \times S_2^2) = \bigoplus_{i=0} ^{n} H_i(S_1^2) \otimes H_{n-i}(S_2^2).$$ Carrying out this calculation, we see 
\begin{enumerate}
\item[] $H_*(S_1^2 \times S_2^2) = H_0(S_1^2 \times S_2^2)\oplus H_2(S_1^2 \times S_2^2) \oplus H_4(S_1^2 \times S_2^2)$ where 
\begin{enumerate}
 \item[] $ H_0(S_1^2 \times S_2^2) = H_0(S_1^2)\otimes H_0(S_2^2) \cong \mathbb{Z}\{ p \} \times \{ p \}$ 
 \item[] $H_2(S_1^2 \times S_2^2) = (H_0(S_1^2)\otimes H_2(S_2^2)) \oplus (H_2(S_1^2)\otimes H_0(S_2^2)) \cong \mathbb{Z}\{ p \} \times c \oplus \mathbb{Z}c \times \{ p \}$
 \item[] $H_4(S_1^2 \times S_2^2) = H_2(S_1^2) \otimes H_2(S_2^2) \cong \mathbb{Z}c\times c$
 \end{enumerate}
 \end{enumerate}
Inductively, we get the result that 
\begin{enumerate}
\item{For $0\leq k \leq n$, $H_{2k}(S_1^2 \times \cdots \times S_n^2)$ is the free abelian group with basis given by all possible ways to choose $c$ in exactly $k$ positions and $\{ p \}$ in all others.}
\item{ $H_{k}(S_1^2 \times \cdots \times S_n^2) = 0$ otherwise.}
\end{enumerate}

Given $a\in B^n$, $S_a$ is diffeomorphic to the cartesian product of $n$ two-spheres, so $H_*(S_a)$ has the structure described above. Recall, that there is one copy of $S^2$ for each arc in the matching. We index the spheres with respect to their corresponding arcs. 

For $a\rightarrow b$, $S_a\cap S_b$ is homeomorphic to the cartesian product of $2n-1$ two-spheres.  To see this, recall the matchings $a$ and $b$ differ by exactly two arcs: $(i,j), (k,l) \in a$ and $(i,l), (j,k) \in b$ where $i<j<k<l$. 

The spheres corresponding to arcs identical in $S_a$ and $S_b$ are all present in $S_a\cap S_b$, and the spheres corresponding to the arcs which are different intersect on their diagonals. In other words, if $S_a = S_{(u_1, v_1)}^2 \times \cdots S_{(u_{n-2}, v_{n-2})}^2 \times S_{(i,j)}^2\times S_{(k,l)}^2$ and $S_b = S_{(u_1, v_1)}^2 \times \cdots S_{(u_{n-2}, v_{n-2})}^2 \times S_{(i,l)}^2\times S_{(j,k)}^2$, then $S_a\cap S_b = S_{(u_1, v_1)}^2 \times \cdots S_{(u_{n-2}, v_{n-2})}^2 \times \{(x,x): x\in S^2 \}$. We call this diagonal sphere $S_{\Delta}^2$. 

For example, given the matching $a_1 =$ \includegraphics[width=.7in]{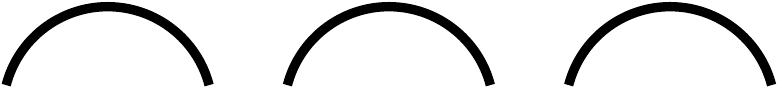}, we will write $S_{a_1} = S_{(1,2)}^2\times  S_{(3,4)}^2 \times S_{(5,6)}^2$, and for the matching $a_2=$ \raisebox{-3pt}{\includegraphics[width=.7in]{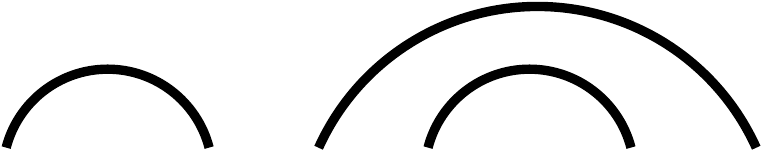}}, we will write $S_{a_2} = S_{(1,2)}^2\times S_{(3,6)}^2\times S_{(4,5)}^2$.   So, $S_{a_1}\cap S_{a_2} = S_{(1,2)}^2\times S_{\Delta}^2$.

The inclusion maps $\psi_{a,b}: S_a\cap S_b\rightarrow S_a$ and $\psi_{b,a}: S_a\cap S_b \rightarrow S_b$ send $$(x_1, \ldots, x_{n-2}, x_\Delta)  \mapsto (x_1, \ldots, x_{n-2}, x_\Delta, x_\Delta)$$ in both $S_a$ and $S_b$.

Let $Q, T$ be basis elements for $H_{2k}(S_{(u_1, v_1)}^2 \times \cdots S_{(u_{n-2}, v_{n-2})}^2)$ and $H_{2k-2}(S_{(u_1, v_1)}^2 \times \cdots S_{(u_{n-2}, v_{n-2})}^2)$ respectively. Every basis element in $H_{2k}(S_a\cap S_b)$ has either the form $Q \times \{ p \}$ or $T \times  c$, and the inclusion maps operate on these elements as follows,
$$Q \times  \{ p \} \mapsto Q \times  \{ p \} \times \{ p \}$$
$$ T \times  c  \mapsto  T \times \{ p \} \times c+T \times c \times \{ p \}.$$

\section{Relating $\mathcal{BN}(A\times I, c_{2n}, \mathbb{Z})$ to $H_*(\widetilde{S})$}\label{isom}

\subsection{The short exact sequence}

In \cite{K2}, Khovanov proves that the following sequence is exact:
$$ 0\rightarrow H^*(\widetilde{S}) \xrightarrow{\phi} \bigoplus_b H^*(S_b) \xrightarrow{\psi ^-} \bigoplus_{a<b} H^*(S_a\cap S_b).$$ The map $\phi$ is induced by the inclusions $S_b\subset \widetilde{S}$ and $\psi^{-} = \sum_{a<b} (\psi_{a,b} - \psi_{b,a})$ where 
$\psi_{a,b}:H^*(S_a)\rightarrow H^*(S_a\cap S_b)$ is induced by the inclusion $(S_b\cap S_a)\subset S_b$ and $\psi_{b,a}$ is analogously defined.

We wish to modify this sequence by restricting the codomain of $\psi^-$ to  $\bigoplus_{a\rightarrow b} H^*(S_a\cap S_b)$. In order to do this and maintain exactness, the kernel of $\psi^-$ must remain unchanged by this restriction. The following Lemma is a necessary tool in proving this fact.

\begin{lemma} \label{khdist} (Khovanov)
Given $a,b,c\in B^n$, if $d(a,c) = d(a, b) + d(b, c)$, then $S_a\cap S_c = S_a\cap S_b \cap S_c$.
\end{lemma} 
Inductively, this means that if $d(a,c) = m$ with $(a=a_0, \ldots a_m=c)$ a corresponding minimal sequence, then $S_a\cap S_c = \bigcap_{i=0}^{m} S_{a_i}$. 

\begin{theorem}
Let $\widetilde{ \psi^-}:  \bigoplus_b H^*(S_b)  \rightarrow \bigoplus_{a\rightarrow b} H^*(S_a\cap S_b)$ be the map obtained by restriction of the codomain of $\psi^-$. Then $Ker(\psi^-)=Ker(\widetilde{ \psi^-})$.
\end{theorem}
\begin{proof}
Given matchings $a<b$, let $\iota_{a,b}$ be the map on homology induced by the inclusion $S_a \cap S_b \subset S_a$. Since $a<b$ whenever $a\rightarrow b$, $Ker(\psi^-) \subset Ker(\widetilde{ \psi^-})$. 

Assume that $x = \sum_{b\in B} x_b \in Ker(\widetilde{ \psi^-})$. Given any $a<b$, we must show that  $\psi_{a,b}(x_a) = \psi_{b,a}(x_b)$. In other words, $x_a\circ \iota_{a,b} = x_b \circ \iota_{b,a}$. This will be proven by induction on $d(a,b)$.

\underline{Base Case}: If $d(a,b) = 1$, then $a\rightarrow b$. Thus, $x_a \circ \iota_{a,b} = x_b \circ \iota_{b,a}$ by definition of $\widetilde{\psi^-}$.

\underline{Inductive Argument}: Assume for every $a<b$ with $d(a,b)<k$, $\psi_{a,b}(x_a) = \psi_{b,a}(x_b)$. Consider $a<b$ with $d(a,b) = k$. Then, there is some $c\in B$ with $a<c<b$, $d(a,c) = k-1$, and $d(c,b) = 1$.  From Lemma \ref{khdist}, $S_a\cap S_b = S_a\cap S_c \cap S_b$.  

Let $$\zeta_{a,c}: H_*(S_a\cap S_b) \rightarrow H_*(S_a\cap S_c) \text{ and } \zeta_{b,c}:H_*(S_a\cap S_b) \rightarrow H_*(S_c\cap S_b)$$ be the maps on homology induced by the inclusion of $S_a\cap S_b$ into $S_a\cap S_c$ and $S_c\cap S_b$.

Since $d(a,c) = k-1$, $x_a\circ \iota_{a,c} = x_c \circ \iota_{c,a}.$ Since $d(c,b)=1$, $x_c\circ \iota_{c,b} = x_b \circ \iota_{b,c}.$ Precomposing with inclusion maps gives the following two equations,  
$$x_a\circ \iota_{a,c} \circ \zeta_{a,c} = x_c \circ \iota_{c,a} \circ \zeta_{a,c}$$
$$x_c\circ \iota_{c,b} \circ \zeta_{b,c} = x_b \circ \iota_{b,c} \circ \zeta_{b,c}.$$ 
Furthermore, $\iota_{c,a}\circ \zeta_{a,c} = \iota_{c,b} \circ \zeta_{b,c}$.   The result is $$x_a\circ \iota_{a,c} \circ \zeta_{a,c} =  x_b \circ \iota_{b,c} \circ \zeta_{b,c}.$$ Using properties of composition once again, $\iota_{a,c} \circ \zeta_{a,c} = \iota_{a,b}$, and $\iota_{b,c} \circ \zeta_{b,c}= \iota_{b,a}.$ Finally, substitution yields $$x_a\circ \iota_{a,b}= x_b\circ \iota_{b,a}.$$ 

Thus, $Ker(\widetilde{\psi^-}) = Ker(\psi^-)$, as desired.  
\end{proof}
Consequently, the following sequence is exact: $$ 0\rightarrow H^*(\widetilde{S}) \xrightarrow{\phi} \bigoplus_b H^*(S_b) \xrightarrow{\widetilde{\psi ^-}} \bigoplus_{a\rightarrow b} H^*(S_a\cap S_b).$$

\subsection{The isomorphism}\label{theiso}

As discussed in Section \ref{const}, $H(\widetilde{S})$, $\bigoplus_b H(S_b)$, and $\bigoplus_{a\rightarrow b} H(S_a\cap S_b)$ are all finitely generated, free abelian groups. The dual sequence on homology is therefore also exact. In other words,
 $$ \bigoplus_{a\rightarrow b} H_*(S_a\cap S_b)\xrightarrow{\widetilde{\psi^-}} \bigoplus_b H_*(S_b) \xrightarrow{\phi} H_*(\tilde{S}) \rightarrow 0$$ is exact where $\phi$ and $\psi^-$ are induced by the same inclusions. 
 
 Examining this sequence, we see that $H_*(\tilde{S}) \cong \bigoplus_b H_*(S_b)/ Im(\psi^-)$. In Section \ref{inc}, we discussed the behavior of the inclusion maps $\psi_{a,b}$ and $\psi_{b,a}$ for $a\rightarrow b$. The generators of $H_*(S_a\cap S_b)$ corresponding to choice of a point in $S_{\Delta}^2$ are mapped to the generators of $H_*(S_a)$ and $H_*(S_b)$ corresponding to choice of a point in both of the spheres that intersect along their diagonal. The generators corresponding to the choice of a two-cell in $S_{\Delta}^2$ are mapped to the sum of the generators corresponding to the choice of a point in one sphere and a two-cell in the other.
 
The relationship between surfaces in the Bar-Natan skein module and generators of $\bigoplus_{b\in B^n} H_*(S_b)$ is described by the rules below. 
\begin{itemize}
\item Each space $S_b$ corresponds to a particular annular configuration. 
\item Each sphere in the cross-product $S_b$ corresponds to a particular annulus in that configuration. \item Recall that the generators for $H_*(S_b)$ correspond to the choice of a point or a two-cell for each sphere. Let choice of a point correspond to a dot on the associated annulus in the configuration and choice of a two-cell correspond to absence of a dot. 
\end{itemize}

As an example of this correspondence, let $b=$ \includegraphics[width=.5in]{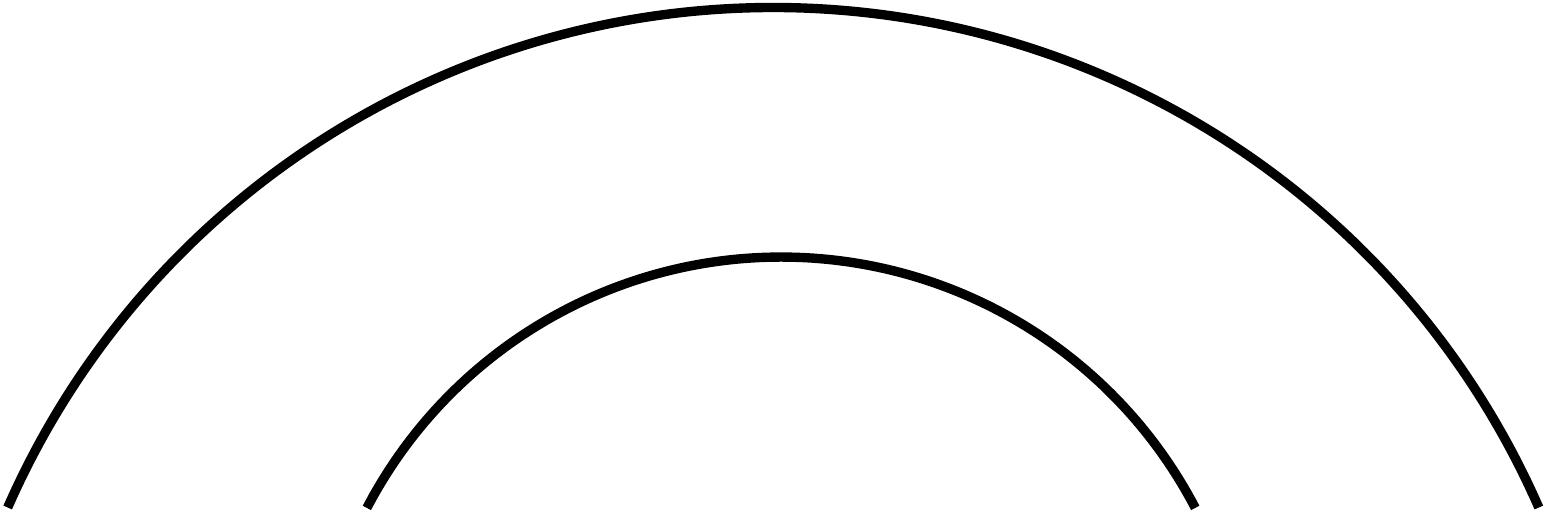}.  Then, $b$ corresponds to the nested annular configuration. The space $S_b$ has the form $ \{ (x,y,y,x): x,y\in S^2\}$, and we can write $S_b = S_{(1,4)}^2 \times S_{(2,3)}^2$. There are four nonzero markings of the nested configuration, and there are four generators of $H_*(S_b)$. We identify them as follows:
\begin{align*}
\includegraphics[width=.6in]{nest.pdf}& \text{ corresponds to } c\times c \\
\includegraphics[width=.6in]{nest2.pdf}& \text{ corresponds to } \{p \} \times c \\
\includegraphics[width=.6in]{nest1.pdf}& \text{ corresponds to } c\times \{ p \} \\
\includegraphics[width=.6in]{nest3.pdf}& \text{ corresponds to } \{ p \} \times \{ p \}
\end{align*}

Using these rules, there is a one to one correspondence between markings of a particular annular configuration and generators of $H_*(S_b)$. Extend this to all matchings to get a one-to-one correspondence between generators of $\bigoplus_{b} H_*(S_b)$ and elements of $\mathcal{F}_{inc}(A\times I, c_{2n})$. 

Now, we wish to compare $Im(\widetilde{\psi^-})$ to $\mathcal{S}(A\times I,c_{2n})$. We have proven that $\mathcal{S}(A\times I, c_{2n})$ is generated by Type I and Type II relations. We will show that the image of $\widetilde{\psi^-}$ is exactly generated by expressions that correspond to Type I and Type II relations.

Recall from Section \ref{inc} that for $a\rightarrow b$, every generator of $H_*(S_a\cap S_b)$ has the form $Q\times \{ p \}$ or $Q\times c$. The cell $Q$ is a generator for the homology of $(S^2)^{n-2}$, the cartesian product of the spheres that $S_a$ and $S_b$ have in common, while the final $\{ p \}$ or $c$ corresponds to the diagonal sphere.   Since the cells $Q\times \{ p \}$ and $Q\times c$ are the generators for $\bigoplus_{a< b} H_*(S_a\cap S_b)$, the image of those elements under $\widetilde{\psi^-}$ will generate $Im(\widetilde{\psi^-})$. 

Another fact shown in Section \ref{inc} is that $$\psi_{a,b}(Q\times \{p \} ) = Q\times \{ p \} \times \{ p \} \in H_*(S_a)$$ and $$\psi_{a,b}(Q\times c) = Q\times \{ p \} \times c + Q \times c\times \{ p \} \in H_*(S_a).$$ The map $\psi_{b,a}$ behaves identically but sends these generators to elements of $H_*(S_b)$. 

Recall that $\widetilde{\psi^-} = \sum_{a\rightarrow b} (\psi_{a,b} - \psi_{b,c})$, so
$$\widetilde{\psi^-}(Q\times \{ p \}) =  Q\times \{ p \} \times \{ p \}  - Q\times \{ p \} \times \{ p \} \in H_*(S_a) \oplus H_*(S_b)$$
$$\widetilde{\psi^-}(Q\times c) = \Big( Q\times \{ p \} \times c + Q \times c\times \{ p \} \Big) - \Big( Q\times \{ p \} \times c + Q \times c\times \{ p \} \Big).$$

In the context of the Bar-Natan skein module, $Q$ represents an arbitrary marked configuration of $n-2$ annuli and is fixed by $\widetilde{\psi^-}$. Extending the diagrammatic correspondence described earlier to the image of these generators gives the following

\begin{equation}\label{diatype2}
\widetilde{\psi^-}(Q\times \{ p \}) = \raisebox{-5pt}{\includegraphics[height=.3in]{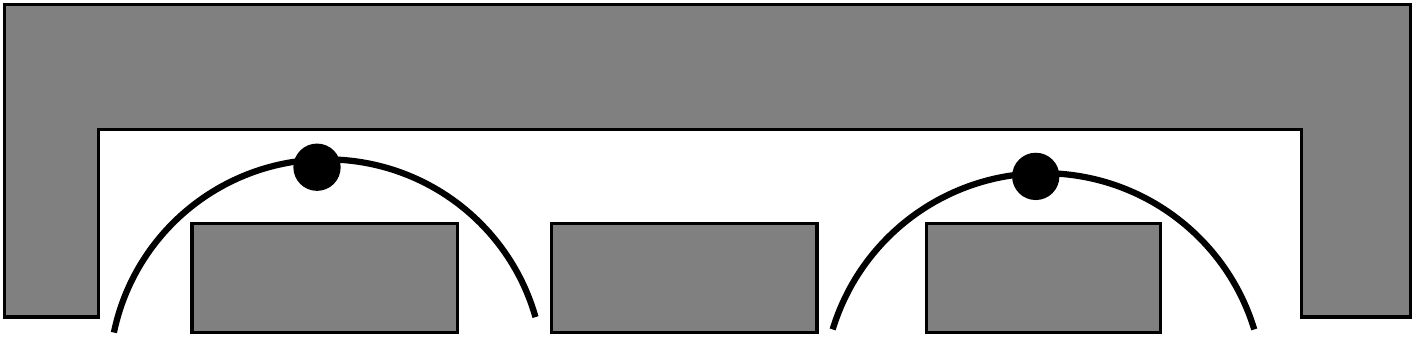}}  -  \raisebox{-5pt}{\includegraphics[height=.3in]{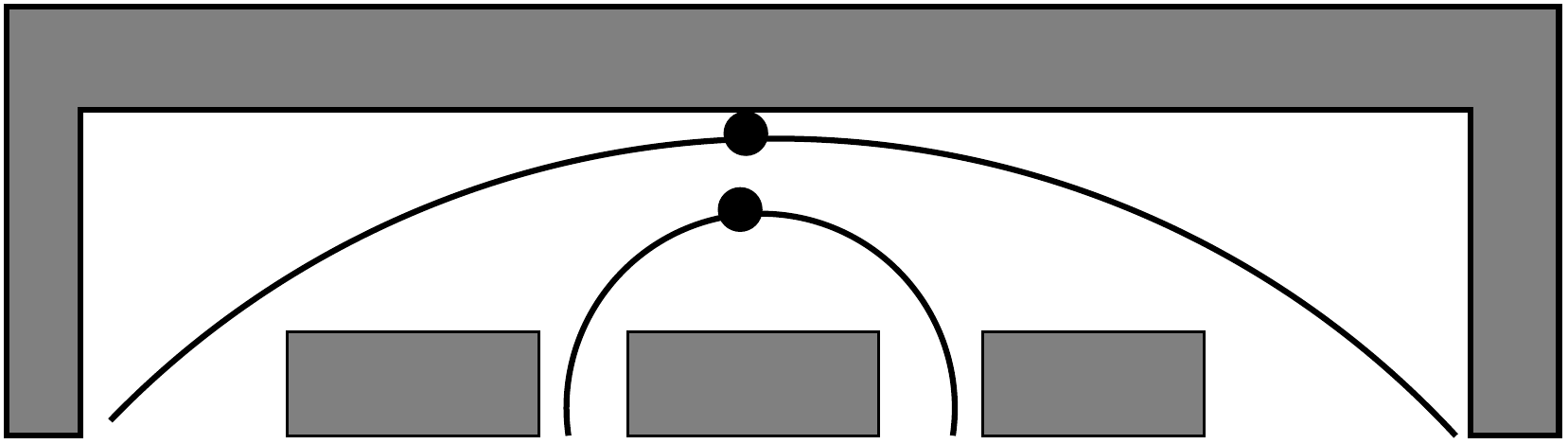}}
\end{equation}
\begin{equation}\label{diatype1}
\widetilde{\psi^-}(Q\times c) = (\raisebox{-5pt}{\includegraphics[height=.27in]{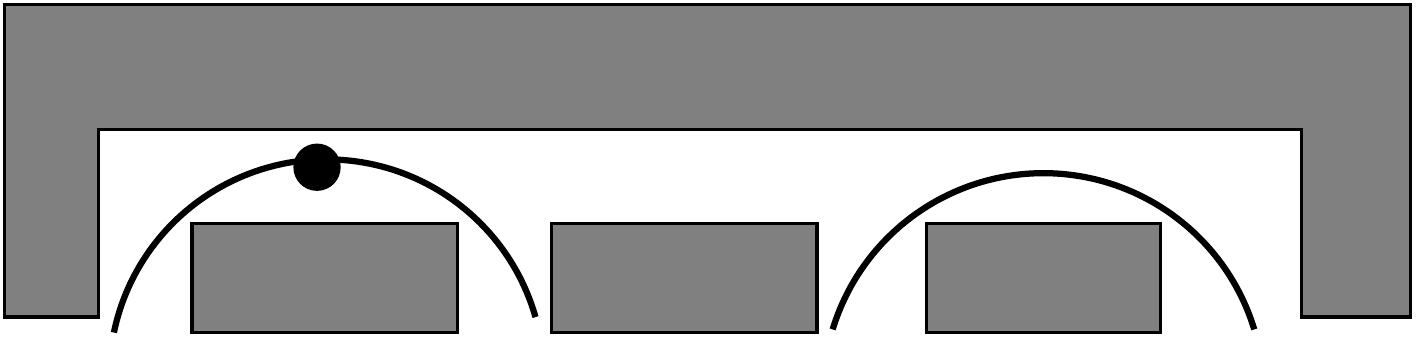}} + \raisebox{-5pt}{\includegraphics[height=.27in]{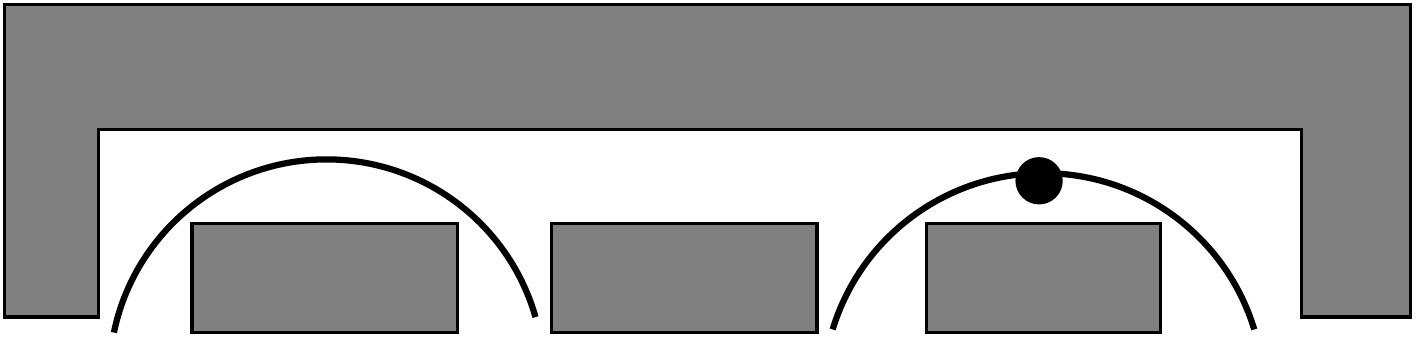}}) -  (\raisebox{-5pt}{\includegraphics[height=.27in]{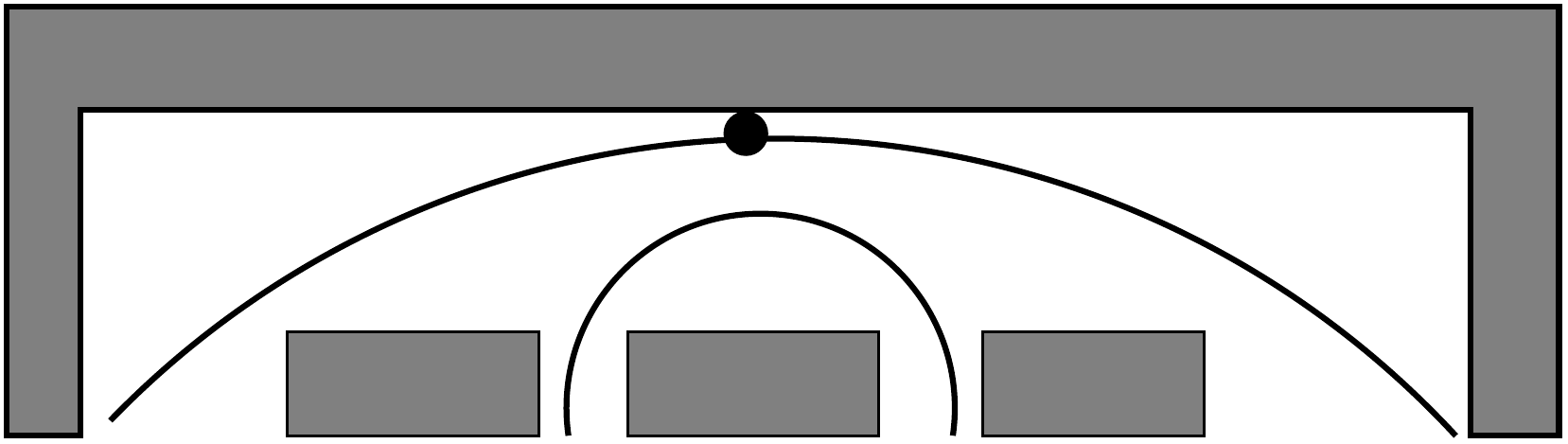}} + \raisebox{-5pt}{\includegraphics[height=.27in]{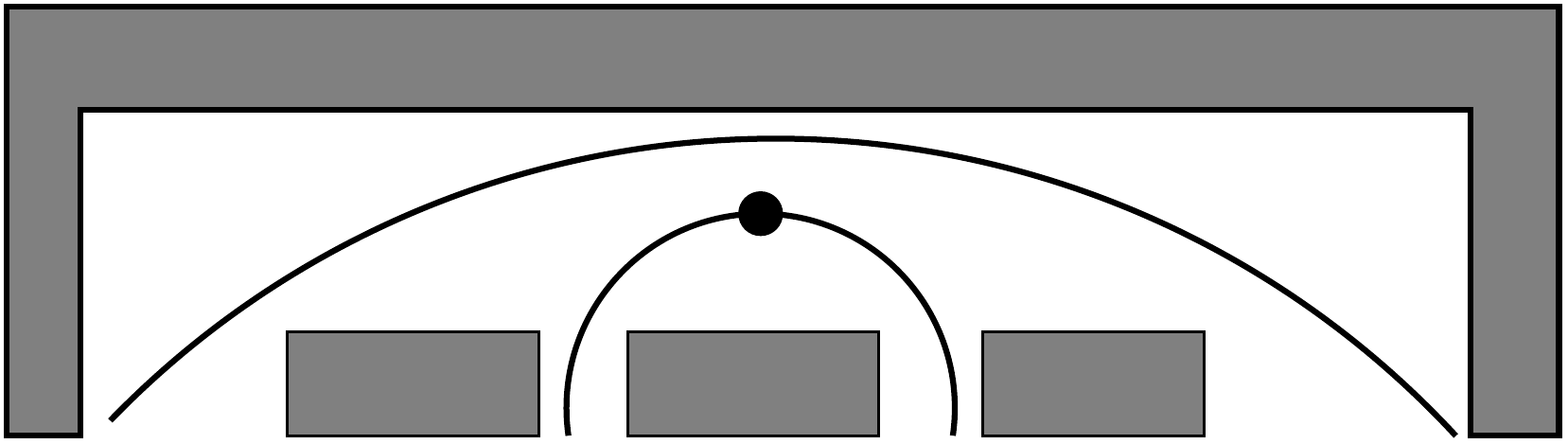}})
\end{equation}

But Eq. \ref{diatype2} is just the diagrammatic version of the Type II relation, and Eq. \ref{diatype1} is the diagrammatic version of the  Type I relation. Thus, there is a one to one correspondence between elements of $Im(\widetilde{\psi^-})$ and $\mathcal{S}(A\times I, c_{2n})$. This allows us to conclude that $\bigoplus_{b} H(S_b)/ Im(\psi^-) \cong \mathbb{Z}\mathcal{F}(A\times I,c_{2n})/ \mathcal{S}(A\times I,c_{2n})$. Via \cite{K2},  both are isomorphic to the total homology of the $(n,n)$ Springer variety of complete flags in $\mathbb{C}^{2n}$ fixed by a nilpotent matrix with two Jordan blocks of size $n$. This completes the proof of the following result:
\begin{maintheorem}
 The Bar-Natan skein module of the solid torus with boundary curve system $2n$ copies of the longitude is isomorphic to the total homology of the $(n,n)$ Springer variety. 
 \end{maintheorem}

\subsection{On degree oddity} 
 The notions of degree in the Bar-Natan skein module and $H_*(\widetilde{S})$ are not preserved by the correspondence of Section \ref{theiso}. Recall that, in the Bar-Natan skein module, the degree of a marked surface $S$ is $2d-\chi (S)$ where $d$ is the number of dots and $\chi$ is the usual Euler characteristic. Since $\mathcal{BN}(A\times I, c_{2n}, \mathbb{Z})$ is generated by disjoint unions of marked annuli, the formula for degree on incompressible surfaces simplifies to deg($S$) = $2d$. A single annulus has degree $0$ if undotted and degree $2$ if it carries one dot. It cannot carry more than one dot because of the relation TD.

In the correspondence of Section \ref{theiso}, a dotted annulus in the Bar-Natan skein module is associated to $\{ p \} \in H_*(S^2)$ while an undotted annulus corresponds to $c\in H_*(S^2)$. Thus, we are mapping elements of degree 2 to elements of degree 0 and vice versa. Any behavior that the Bar-Natan skein module might inherit from $H_*(\widetilde{S})$ and thus from the homology of the $(n,n)$ Springer variety will, therefore, be upside down since we are matching up opposing elements.

To get an idea of why this is necessary, consider the graded rank of $H^*(\widetilde{S})$ versus that of the Bar-Natan skein module. By Lemma \ref{alldot}, there is only one generator for the Bar-Natan skein module in degree $2n$. There are $C_n = \frac{1}{n+1}\binom{2n}{n}$ generators in degree 0 since there are no relations between undotted incompressible surfaces. On the other hand, $H_*(\widetilde{S})$ has only one generator in degree 0 because $\widetilde{S}$ is connected.  This degree oddity could be fixed by defining the degree of marked, incompressible surfaces in $\mathcal{BN}(A\times I, c_{2n}, \mathbb{Z})$ to be $2n -2d$. We choose not to do this because it disagrees with the existing literature.
\section{Algebraic properties of the Bar-Natan skein module}\label{alg}
In \cite{K1}, Khovanov defines a graded ring $H^n$ that is naturally isomorphic to $\bigoplus_{a} H^*(S_a)$ and in \cite{K2} proves that the cohomology of the $(n,n)$ Springer variety is isomorphic to the center of $H^n$.  Multiplication in the center of $H^n$ can be expressed as an operation on dotted crossingless matchings. In this section, we give a comultiplication in the Bar-Natan skein module that can also be stated using the language of marked crossingless matchings.

The algebraic structure described in \cite{K2} together with the comultiplication defined here might seem to hint at a Frobenius structure within the homology of the $(n,n)$ Springer variety. In this section, we explain why this fails to be true.
\subsection{Comultiplication}

The goal of this section is to prove the following 
 \begin{theorem}\label{comult}
 The Bar-Natan skein module $\mathcal{BN}(A\times I, c_{2n}, \mathbb{Z})$ has a well-defined comultiplication.
 \end{theorem}
Given a toplogical space $X$, the diagonal map $\widetilde{\Delta} : X\rightarrow X\times X$ induces a map on homology
$$\widetilde{\Delta} : H_*(X) \rightarrow H_*(X\times X)$$
 In Section \ref{inc}, we discussed how this map behaves when $X=S^2$. Recall that the cells $\{ p \}$ and $c$ are generators of homology in $S^2$ and $\widetilde{\Delta}(\{ p \}) = \{ p \} \times \{ p \}$ while $\widetilde{\Delta} (c) = \{ p \} \times c + c\times \{ p \}$. This extends to a diagonal mapping on $H_*(S_a)$ in the natural way, since $S_a$ is diffeomorphic to the cartesian product of $n$ two-spheres.
 
Recall that the Eilenberg-Zilber map on torsion-free homology is an isomorphism
$$EZ: H_*(X\times X) \rightarrow H_*(X)\otimes H_*(X)$$

Putting the Eilenberg-Zilber map together with the diagonal map on homology, we get a comultiplication on the homology of $S_a$.
$$\Delta :H_*(S_a) \rightarrow H_*(S_a)\otimes H_*(S_a)$$

This comultiplication mapping can be represented using dotted, crossingless matchings as shown below. Recall that each arc in the matching $a$ represents a sphere in the cartesian product $S_a$. Dots on arcs represent a choice of $\{ p \}$ in the cartesian product cell decomposition, and undotted arcs represent a choice of $c$.

Consider the unnested matching on 4 nodes. Then the comutiplication map behaves as follows
\begin{eqnarray*}
\Delta(\includegraphics[width=.5in]{unnest3.pdf} )&=& (\includegraphics[width=.5in]{unnest3.pdf} \otimes \includegraphics[width=.5in]{unnest3.pdf}) \\
\Delta(\includegraphics[width=.5in]{unnest1.pdf} )&=& (\includegraphics[width=.5in]{unnest1.pdf} \otimes \includegraphics[width=.5in]{unnest3.pdf}) +  (\includegraphics[width=.5in]{unnest3.pdf} \otimes \includegraphics[width=.5in]{unnest1.pdf}) \\
\Delta(\includegraphics[width=.5in]{unnest2.pdf} )&=& (\includegraphics[width=.5in]{unnest2.pdf} \otimes \includegraphics[width=.5in]{unnest3.pdf}) +  (\includegraphics[width=.5in]{unnest3.pdf} \otimes \includegraphics[width=.5in]{unnest2.pdf}) \\
\Delta(\includegraphics[width=.5in]{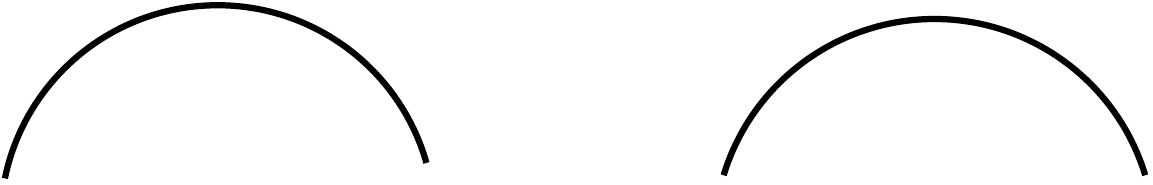} )&=& (\includegraphics[width=.5in]{unnest3.pdf} \otimes \includegraphics[width=.5in]{revunnest0.pdf})+ (\includegraphics[width=.5in]{unnest1.pdf} \otimes \includegraphics[width=.5in]{unnest2.pdf}) + (\includegraphics[width=.5in]{unnest2.pdf} \otimes \includegraphics[width=.5in]{unnest1.pdf}) + (\includegraphics[width=.5in]{revunnest0.pdf} \otimes \includegraphics[width=.5in]{unnest3.pdf}) 
\end{eqnarray*}

Extend this map linearly to $\bigoplus_{a\in B^n} H_*(S_a)$. After composing with inclusion into the appropriate tensor product,  the resulting map, denoted again by $\Delta$, is
$$\Delta: \bigoplus_{a\in B^n} H_*(S_a) \rightarrow \bigoplus_{a\in B^n} (H_*(S_a) \otimes H_*(S_a)) \hookrightarrow (\bigoplus_{a\in B^n} H_*(S_a)) \otimes (\bigoplus_{a\in B^n} H_*(S_a))$$

Recall the following exact sequence from Section \ref{theiso}
$$ \bigoplus_{a\rightarrow b} H_*(S_a\cap S_b)\xrightarrow{\widetilde{\psi^-}} \bigoplus_a H_*(S_a) \xrightarrow{\phi} H_*(\widetilde{S}) \rightarrow 0.$$

We wish to prove that the map $\Delta$ is well-defined on the quotient of $\bigoplus_{a\in B^n} H_*(S_a)$ by $Im(\widetilde{\psi})$. If this is a well-defined mapping, then there is a well-defined comultiplication on the Bar-Natan skein module.  

We prove this by showing that everything in $Im(\widetilde{\psi})$ is mapped by $\Delta$ to $$Im(\widetilde{\psi}) \otimes (\bigoplus_{a\in B^n} H_*(S_a)) + (\bigoplus_{a\in B^n} H_*(S_a)) \otimes Im(\widetilde{\psi}).$$ 

Theorem \ref{reltypes} guarantees that every relation between incompressible surfaces in the Bar-Natan skein module is a linear combination of Type I and Type II relations. Each of these relations involves interaction only between two adjacent arcs.  Thus, it is enough to prove that $\Delta$ is well-defined in the $n=2$ case.

In this case, $Im(\widetilde{\psi})$ is generated by two expressions
\begin{equation*}
\raisebox{-5pt}{\includegraphics[width=.6in]{unnest1.pdf}} + \raisebox{-5pt}{\includegraphics[width=.6in]{unnest2.pdf}} \hspace{.1in} - \hspace{.1in} \raisebox{-10pt}{ \includegraphics[width=.6in]{nest1.pdf}} - \raisebox{-10pt}{\includegraphics[width=.6in]{nest2.pdf}}
\end{equation*}
\begin{equation*}
\raisebox{-5pt}{\includegraphics[width=.6in]{unnest3.pdf}} \hspace{.1in} - \raisebox{-10pt}{ \includegraphics[width=.6in]{nest3.pdf}}
\end{equation*}

Applying $\Delta$ to the first expression, we get
\begin{eqnarray*}
\Delta(\includegraphics[width=.5in]{unnest1.pdf} &+& \includegraphics[width=.5in]{unnest2.pdf} \hspace{.1in} - \hspace{.1in}  \includegraphics[width=.5in]{nest1.pdf} - \includegraphics[width=.5in]{nest2.pdf})  =  \\ 
(&\includegraphics[width=.5in]{unnest1.pdf}& \otimes \includegraphics[width=.5in]{unnest3.pdf}) +  (\includegraphics[width=.5in]{unnest3.pdf} \otimes \includegraphics[width=.5in]{unnest1.pdf}) 
 + (\includegraphics[width=.5in]{unnest2.pdf} \otimes \includegraphics[width=.5in]{unnest3.pdf}) +  (\includegraphics[width=.5in]{unnest3.pdf} \otimes \includegraphics[width=.5in]{unnest2.pdf}) \\
-(&\includegraphics[width=.5in]{nest3.pdf}& \otimes \includegraphics[width=.5in]{nest1.pdf})  -  
(\includegraphics[width=.5in]{nest1.pdf} \otimes \includegraphics[width=.5in]{nest3.pdf})
-(\includegraphics[width=.5in]{nest2.pdf} \otimes \includegraphics[width=.5in]{nest3.pdf}) -  (\includegraphics[width=.5in]{nest3.pdf} \otimes \includegraphics[width=.5in]{nest2.pdf})
\end{eqnarray*}

By adding zero in the appropriate way, we get
\begin{eqnarray*}
 (\includegraphics[width=.5in]{unnest1.pdf}+\includegraphics[width=.5in]{unnest2.pdf}) \otimes \includegraphics[width=.5in]{unnest3.pdf} &+& \includegraphics[width=.5in]{unnest3.pdf} \otimes (\includegraphics[width=.5in]{unnest1.pdf}+\includegraphics[width=.5in]{unnest2.pdf}) \\
- (\includegraphics[width=.5in]{nest1.pdf}+\includegraphics[width=.5in]{nest2.pdf}) \otimes \includegraphics[width=.5in]{nest3.pdf} &-& \includegraphics[width=.5in]{nest3.pdf} \otimes (\includegraphics[width=.5in]{nest1.pdf}+\includegraphics[width=.5in]{nest2.pdf}) \\
+ (\includegraphics[width=.5in]{nest1.pdf}+\includegraphics[width=.5in]{nest2.pdf})\otimes \includegraphics[width=.5in]{unnest3.pdf} &-& (\includegraphics[width=.5in]{nest1.pdf}+\includegraphics[width=.5in]{nest2.pdf})\otimes \includegraphics[width=.5in]{unnest3.pdf} \\
+ \includegraphics[width=.5in]{unnest3.pdf} \otimes  (\includegraphics[width=.5in]{nest1.pdf}+\includegraphics[width=.5in]{nest2.pdf}) &-& \includegraphics[width=.5in]{unnest3.pdf} \otimes  (\includegraphics[width=.5in]{nest1.pdf}+\includegraphics[width=.5in]{nest2.pdf})
\end{eqnarray*}

Regrouping gives
\begin{eqnarray*}
(\text{Type I Relation}) \otimes  \includegraphics[width=.5in]{unnest3.pdf} &+&  \includegraphics[width=.5in]{unnest3.pdf} \otimes (\text{Type I Relation})\\
+ (\text{Type II Relation}) \otimes (\includegraphics[width=.5in]{nest1.pdf}+\includegraphics[width=.5in]{nest2.pdf}) &+& (\includegraphics[width=.5in]{nest1.pdf}+\includegraphics[width=.5in]{nest2.pdf})\otimes (\text{Type II Relation})
\end{eqnarray*}

Applying $\Delta$ to the second expression, we get: 
$$\Delta(\includegraphics[width=.5in]{unnest3.pdf} - \includegraphics[width=.5in]{nest3.pdf})  =   
\includegraphics[width=.5in]{unnest3.pdf} \otimes \includegraphics[width=.5in]{unnest3.pdf} -  \includegraphics[width=.5in]{nest3.pdf} \otimes \includegraphics[width=.5in]{nest3.pdf}$$

Once again, adding zero in the appropriate way gives
\begin{eqnarray*}
\includegraphics[width=.5in]{unnest3.pdf} \otimes \includegraphics[width=.5in]{unnest3.pdf} -  \includegraphics[width=.5in]{nest3.pdf} \otimes \includegraphics[width=.5in]{nest3.pdf} &+& \includegraphics[width=.5in]{unnest3.pdf} \otimes \includegraphics[width=.5in]{nest3.pdf} -  \includegraphics[width=.5in]{unnest3.pdf} \otimes \includegraphics[width=.5in]{nest3.pdf} = \\
\includegraphics[width=.5in]{unnest3.pdf} \otimes (\includegraphics[width=.5in]{unnest3.pdf} - \includegraphics[width=.5in]{nest3.pdf}) &+& (\includegraphics[width=.5in]{unnest3.pdf} - \includegraphics[width=.5in]{nest3.pdf}) \otimes \includegraphics[width=.5in]{nest3.pdf}  = \\
\includegraphics[width=.5in]{unnest3.pdf} \otimes (\text{Type II Relation}) &+& (\text{Type II Relation}) \otimes \includegraphics[width=.5in]{nest3.pdf} 
\end{eqnarray*}
 
 From these calculations, we conclude that $\Delta(\text{Im}(\widetilde{\psi})) \subseteq \text{Im}(\widetilde{\psi}) \otimes (\bigoplus_{a\in B^n} H_*(S_a)) + (\bigoplus_{a\in B^n} H_*(S_a)) \otimes \text{Im}(\widetilde{\psi})$, and so Theorem \ref{comult} is proven. 

  \subsection{Obstructions to a Frobenius structure}
A cohomology ring admits a Frobenius structure if and only if its associated homology ring admits one. Additionally, any manifold with cohomology only in even dimensions and with no torsion admits a Frobenius structure.

The $(n,n)$ Springer variety has torsion-free cohomology lying only in even dimensions, but it is not a manifold. One might hope that its cohomology ring might still be a Frobenius extension of $\mathbb{Z}$, however this turns out not to be true. 

From \cite{K2}, the cohomology ring of the $(n,n)$ Springer variety has the following polynomial ring presentation
$$ S = \mathbb{Z}[x_1, \ldots , x_{2n}] / R_1 \cong H^*(B(n,n), \mathbb{Z}) $$ where $R_1$ is generated by $x_i^2$ for all $1\leq i \leq 2n$ and the elementary symmetric polynomials  of degree $1\leq k\leq 2n$. All variables are taken to have degree 2. Also from \cite{K2}, any basis $\{v_1, \ldots , v_m\}$ for $S$ as a $\mathbb{Z}$-module has two generators in degree $n$, the top degree. Let these two generators be denoted $v$ and $v'$.  There is only one generator in degree 0, namely the element 1. 

If $S$ were a Frobenius system with respect to some linear functional $\epsilon : S \rightarrow \mathbb{Z}$, then there would exist a dual basis $\{ u_1, \ldots u_m \}$ such that $\epsilon (u_i v_j) = \delta_{ij}$. Let $u$ and $u'$ be the elements of the dual basis paired with $v$ and $v'$ respectively.  Given any polynomial $f$, the products $fv$ and $fv'$ are nonzero in $S$ if and only if $f$ is a scalar. This is a result of $v$ and $v'$ being top-dimensional. Since $u$ and $u'$ cannot both be scalars,  either $\epsilon (uv) = 0$ or $\epsilon (u'v') = 0$. The result is that, regardless of choice of linear functional $\epsilon$ and basis $\{v_1, \ldots , v_m\}$ for $S$, there does not exist a dual basis with the required properties. 

Since a dual basis can never exist, the ring $S\cong H^*(B(n,n))$ is not a Frobenius extension. Therefore, $H_*(B(n,n))$ is also not a Frobenius extension. The result is that the Bar-Natan skein module is not a Frobenius extension as viewed with the structure coming from homology.


\begin{thebibliography}{20}

\bibitem{AF}
M.\ Asaeda and C.\ Frohman,
\textit{A note on the Bar-Natan skein module},
math.GT/0602262v1.

\bibitem{UK} 
U. \ Kaiser, 
{\em Frobenius algebras and skein modules of surfaces in $3$-manifolds},
math.GT/ 0802.4068v1.

\bibitem{K2} M. Khovanov, {\em Crossingless matchings and the (n,n) Springer variety}, Communications in Contemporary Math. 6 (2004) no.2 561--577, math.QA/0103190.

\bibitem{K1} 
M.~Khovanov, {\em A functor-valued invariant of tangles}, Alg. Geom. Top. v.~2 
(2002) 665-741, math.QA/0103190. 

 

\end{thebibliography}
\end{document}